\documentclass[11point]{amsart}
\usepackage{amsmath, xcolor, datetime}
\usepackage{amscd}

\usepackage{amssymb,amsmath,amscd,epsfig,amsthm,enumerate,import,graphicx}
\usepackage{caption}
\usepackage{subcaption}

\usepackage[utf8]{inputenc}

\usepackage{relsize}

\usepackage{hyperref}   
\hypersetup{
	colorlinks=true,  
	pdfborder={0 0 0},      
	pdfauthor={I am the Author} 
}

\usepackage{geometry}                
\newtheorem{theorem}{Theorem}[section]
\newtheorem{proposition}[theorem]{Proposition}
\newtheorem{lemma}[theorem]{Lemma}
\newtheorem{corollary}[theorem]{Corollary}
\theoremstyle{definition}
\newtheorem{definition}[theorem]{Definition}

\theoremstyle{remark}
\newtheorem{remark}[theorem]{Remark}

\newcommand{\be}{\begin{equation}}
	\newcommand{\ee}{\end{equation}}

\makeatletter
\newcommand{\tpitchfork}{%
	\vbox{
		\baselineskip\z@skip
		\lineskip-.52ex
		\lineskiplimit\maxdimen
		\m@th
		\ialign{##\crcr\hidewidth\smash{$-$}\hidewidth\crcr$\pitchfork$\crcr}
	}%
}
\makeatother

\newcommand{\B}{{\bf B}}

\newcommand{\R}{{\bf R}}

\newcommand{\n}{{\bf n}}

\newcommand{\h}{\hbar}
\newcommand{\zbar}{\overline{z}}
\newcommand{\wbar}{\overline{w}}

\newcommand{\bbC}{{\mathbb C}}
\newcommand{\bbS}{{\mathbb S}}

\newcommand{\bbZ}{{\mathbb Z}}
\newcommand{\bbN}{{\mathbb N}}
\newcommand{\bbR}{{\mathbb R}}

\newcommand{\calA}{{\mathcal A}}

\newcommand{\calD}{{\mathcal D}}
\newcommand{\calB}{{\mathcal B}}

\newcommand{\calQ}{{\mathcal Q}}
\newcommand{\calZ}{{\mathcal Z}}
\newcommand{\calK}{{\mathcal K}}
\newcommand{\calL}{{\mathcal L}}

\newcommand{\calH}{{\mathcal H}}

\newcommand{\calF}{{\mathcal F}}
\newcommand{\calC}{{\mathcal C}}
\newcommand{\calR}{{\mathcal R}}
\newcommand{\calS}{{\mathcal S}}
\newcommand{\calW}{{\mathcal W}}

\newcommand{\calO}{{\mathcal O}}

\newcommand{\calI}{{\mathcal I}}

\newcommand{\tr}{{\mbox{Tr}}}

\newcommand{\zetabar}{\overline{\zeta}}

\newcommand{\frakR}{{\mathfrak R}}
\newcommand{\frakS}{{\mathfrak S}}
\newcommand{\frakD}{{\mathfrak D}}
\newcommand{\frakB}{{\mathfrak B}}

\newcommand{\ad}{\text{ad}}
\newcommand{\inner}[2]{\langle#1,#2\rangle}
\newcommand{\norm}[1]{\| #1\|}

\newcommand{\delbar}{\overline{\partial}}

\newcommand{\del}{\partial}

\newcommand{\ave}{\text{\scriptsize av}}

\newcommand{\PB}[2]{\{#1,#2\}}
\reversemarginpar

\DeclareSymbolFont{bbold}{U}{bbold}{m}{n}
\DeclareSymbolFontAlphabet{\mathbbold}{bbold}

\newcommand{\C}{\bbC}
\newcommand{\N}{\bbN}

\def\C{\mathbb{C}}
\def\S{\mathbb{S}^2}
\def\B{\mathbb{B}}
\def\N{\mathbb{N}}

\def\Z{\mathbb{Z}}
\def\calL{\mathcal{L}}
\def\R{\mathbb{R}}
\def\Rcal{\mathcal{R}}
\def\Ical{\mathcal{I}}

\def\f{\varphi}
\def\calZ{\mathcal{Z}}
\newcommand{\cohe}{\alpha_z^k}
\newcommand{\ecohe}{\alpha_z^k}
\newcommand{\dn}{\Lambda_q}

\newcommand{\DS}{\Delta_{\S}}


\begin{document}
	
	\title[Spectral cluster asymptotics]{Spectral cluster asymptotics of the Dirichlet to Neumann operator on the two-sphere}

\author{S. P\'erez-Esteva}
\address{Instituto de Matem\'aticas, UNAM, Unidad Cuernavaca}
\email{spesteva@im.unam.mx}
\thanks{S. P\'erez-Esteva partially supported by the project PAPIIT-UNAM IN104224}
\author{A. Uribe}
\address{Mathematics Department\\
	University of Michigan\\Ann Arbor, Michigan 48109}
\email{uribe@umich.edu}
\author{C. Villegas-Blas}
\address{Instituto de Matem\'aticas, UNAM, Unidad Cuernavaca and ``Laboratorio Solomon Lefschetz" Mexico, Unidad Mixta Internacional del CNRS, Cuernavaca.}
\email{villegas@matcuer.unam.mx}
\thanks{C. Villegas-Blas partially supported by projects CONACYT Ciencia B\'asica  CB-2016-283531-F-0363 and PAPIIT-UNAM  IN116323}

	\begin{abstract}  
	We study the spectrum of the Dirichlet to Neumann operator of the two-sphere
    associated to a Schr\"odinger operator in the unit ball.
	The spectrum forms clusters of size $O(1/k)$ around the sequence of
	natural numbers $k=1,2,\ldots$, and we compute the first three terms in the
	asymptotic distribution of the eigenvalues within the clusters, as $k\to\infty$ (band invariants).
There are two independent aspects of the proof.  The first is a 
study of the Berezin symbol of the Dirichlet to Neumann operator, which 
arises after one applies the averaging method.
The second is the use of a symbolic calculus of Berezin-Toeplitz operators on the manifold
of closed geodesics of the sphere.  
\end{abstract}

	
	\maketitle
	\date{}

	\tableofcontents

		\section{Introduction}
	
	\subsection{Setting and background}

	Let $\Omega\subset\bbR^d$, $d\geq{2}$, be an open, bounded region with smooth boundary, 
	and $q\in C^\infty(\overline\Omega)$.  Consider the  Schr\"odinger operator 
	\begin{equation*}
		L_q=-\Delta +q, \qquad \Delta = \sum_{j=1}^d \frac{\partial^2\ }{\partial x_j^2}.
	\end{equation*}
	We will always assume that $0\in\bbR$ is not in the Dirichlet spectrum of $L_q$, which is the case for example if $q\geq 0$. 
	Then for every $F\in H^{-1}(\Omega)$ there exists a unique weak solution in  $H_0^{1}(\Omega)$ of
	\begin{equation}\label{ISE}
		-\Delta u+qu=F,
	\end{equation}
	where  $H^{-1}(\Omega)$  and   $H_0^{1}(\Omega)$  are,  respectively,  the Sobolev spaces in $\Omega$ of order $-1$ and of order $1$ with vanishing boundary trace. In this case we denote 
	\begin{equation}\label{defresolvente}
		u=\Rcal_q F.
	\end{equation}
	We have that
	\begin{equation*}
		\Rcal_q: C_c^\infty (\Omega)\rightarrow C^\infty (\Omega),
	\end{equation*}
	is continuous since it is bounded in $L^2(\Omega)$.
	
	The solvability of \eqref{ISE} implies that the Dirichlet problem
	\begin{equation}\label{Sch}
		L_q u(x)=0,\quad x\in\Omega, 
	\end{equation}
	\begin{equation}\label{BC}
		u(\xi)=f(\xi),\quad \xi\in \partial\Omega,
	\end{equation}
	can be solved for $f\in H^{1/2}(\partial\Omega)$, 
	where for $s$ real $H^s(\partial\Omega)$ denotes the Sobolev space of order $s$ in $\partial\Omega$.
	This is a classical matter, and the proof is as follows:
	Every $f\in H^{1/2}(\partial\Omega)$ is the trace of some function $v\in H^1(\Omega)$. 
	Fix any such $v$ and let  $F=L_q(v)$, so that $F\in H^{-1}(\Omega)$.
	Then $u=v -\Rcal_q(F)$ is the desired solution of \eqref{Sch} and \eqref{BC}.
	
	The previous discussion justifies the following:
	\begin{definition}  Assume that zero is not in the spectrum of $-\Delta+q$, and
		denote by $\n$ the outward-pointing unit normal vector field along $\partial\Omega$.
		Then the Dirichlet to Neumann (D-N) operator $\Lambda_q$ for the Schr\"odinger operator is the operator
        on $\partial\Omega$ 
defined by
		\begin{equation*}
			\forall f\in H^{1/2}(\partial\Omega) \qquad \Lambda_q(f) =\frac{\partial u}{\partial \n}, 
		\end{equation*}
		where $u$ satisfies \eqref{Sch} and \eqref{BC}.
	\end{definition}
	The Dirichlet to Neumann operator $\dn$ has a long and important history.  The Calder\'on problem  asks for the injectivity of the mapping $q\rightarrow \Lambda_q$.  This problem,  stated originally for the conductivity equation in a region of $\R^3$  by A. Calder\'on, has been greatly extended and developped;
    see for example the excellent survey by G. Uhlmann \cite{Uhl}. 
	Another topic of great interest is the study of the rigidity of the so called Steklov spectrum, which is the spectrum of $\Lambda_0$ (i.e.\ the case $q=0$).  
    This active area has now an extensive  literature (see \cite{GirPolt} for an account of results, problems and references).

\medskip
    In this paper we consider the Dirichlet to Neumann operator  $\Lambda_q$ on 
 $\S=\partial\B$, with $\B$ the unit ball in $\R^3$ and $q\in C^\infty(\overline{\B}).$ 
	The goal of this paper is to study spectral asymptotics of $\Lambda_q$  in the context of the work done by  A. Weinstein, V. Guillemin and one of the authors
 on the spectral theory for the Schr\"{o}dinger operator  in the sphere \cite{W,G,Uribe}.  Specifically, we will calculate the so called \textit{Band Invariants} up to order 2 (see Section \ref{sec:main results}).  
 
We mention that, recently, Barcel\'o et al  in \cite{Barce} studied the so called \textit{Born approximation} for the potential  $q$ of  Calder\'on's problem in the ball, which turns out to be closely related to the spectrum of $\Lambda_q$.

As is usual in the literature, we will use  $\Psi$DO to abbreviate  pseudodifferential operator. 
	
\medskip
	For the following, see \cite{NSU} or \cite{LU}.
	\begin{theorem}\label{backgroundThm}
		$\Lambda_q$ is a $\Psi$DO of order one
		whose principal symbol is the Riemannian norm function on $T^*\partial\Omega\setminus\{0\}$.
		Moreover,
		\begin{equation}\label{eq:perturbation}
			\Lambda_q = \Lambda_0 + S
		\end{equation}
		where $S$ is a $\Psi$DO of order (-1) and $\Lambda_0$ is $\Lambda_q$ with $q\equiv 0$.
		The principal symbol $\sigma_S: T^* {\partial\Omega}\setminus\{0\}\to\bbR $ of $S$ is 
		\begin{equation}\label{}
			\sigma_S(x,\xi) = \frac{q(x)}{2|\xi|}.
		\end{equation}
	\end{theorem}
	In particular, we have that
	$$\Lambda_q: H^{1/2}(\partial\Omega)\rightarrow H^{-1/2}(\partial\Omega).$$

	\subsection{The main results}\label{sec:main results}
	Recall that we are 
 considering  the case when $\Omega=\B$
	so that $\partial\Omega=\S$ is the unit sphere.  The orthogonal group O$(3)$ acts
	on $\B$ and commutes with the Laplacian:
	if $T\in O(n)$  then $\Delta (u\circ T)=\Delta u\circ T$, and also
	$\frac{ \partial u\circ T}{\partial n} (\xi)=\frac{ \partial u}{\partial n} (T\xi)$, for $\xi\in \S$.   
	It follows that
	\begin{equation}\label{DNT}
		\Lambda_{q\circ T}(f\circ T)=(\Lambda_{q}f)\circ T
	\end{equation}
	and
	\begin{equation}\label{elemmatT}
		\langle \Lambda_{q}f,f\rangle_{L^2(\S)}=\langle \Lambda_{q\circ T}(f\circ T),f\circ T\rangle_{L^2(\S)}.
	\end{equation}
	A central role in this paper will be played by the decomposition
	\begin{equation}\label{eq:sphHarm}
		L^2(\S) = \bigoplus_{k=0}^\infty \calH_k,
	\end{equation}
	where $\calH_k$ is the space of spherical harmonics of order $k$.
	To be precise, $\calH_k$ consists of the restrictions to $\S$ of harmonic homogeneous polynomials
	on $\bbR^3$ of degree $k$.  Its dimension is $d_k= 2k+1$. 
	These are also the eigenspaces of the spherical Laplace-Beltrami operator $\DS$, the 
	corresponding eigenvalue being $k(k+1)$.
We will denote by  $\Pi_k$ the orthogonal projector from $L^2(\S)$  onto the space of spherical harmonics $\calH_k$. 
 
 Since the extension to $\B$ of a spherical harmonic $Y\in\calH_k$ is  the solid spherical harmonic $Y(rx)=r^k Y(x)$, $0\leq r\leq 1$, $x\in\S$,  
 then obviously $\Lambda_0 Y = k Y$.  
 We record this observation for future use:
	\begin{proposition}
		For $\S$, the operator $\Lambda_0$ preserves the decomposition (\ref{eq:sphHarm}), and in fact
		\begin{equation}\label{eq:lambdazero}
			\forall k\qquad \Lambda_0|_{\calH_k} = \text{multiplication by }k.
		\end{equation}
	\end{proposition}
	Since $\calH_k$ is an eigenspace of the Laplace-Beltrami operator $\DS$ of $\S$ with eigenvalue
	$k(k+1)$, it follows that
	\[
	\Lambda_0 = \sqrt{\DS +\frac 14} - \frac 12.
	\]
	
	From now on we fix $q\in C^\infty(\overline{\B})$ such that zero is not a Dirichlet eigenvalue of $-\Delta +q$.
	Since $S= \Lambda_{q}-\Lambda_0$ has order $(-1)$, it maps $L^2(\S)\to H^1(\S)$.  This, together with a perturbation 
	argument, implies the following (see Appendix C):
	\begin{theorem}\label{teoremaclusters} 
		There exist a constant $C$ such that the spectrum of  
		$\Lambda_q$ is contained in the union of intervals $\bigcup_{k=0}^{\infty} \;  \left[k-\frac Ck,k+\frac Ck\right]$. Moreover, for k sufficiently large, the interval $ \left[k-\frac Ck,k+\frac Ck\right]$ contains precisely $d_k=2k+1$ eigenvalues of $\Lambda_q$
		counted with multiplicities. 
	\end{theorem}

	Accordingly, for $k$ sufficiently large we will write 
    the eigenvalues with multiplicities of $\Lambda_q$ in the form
	\begin{equation}\label{eq:evalsLambdaq}
		\lambda_{k j}:= k+\mu_{k,j}, \qquad j=1,\ldots, d_k=(2k+1).
	\end{equation}
	Note that, by the previous theorem, $\forall j,\, k\ |\mu_{k,j}| = O (1/k)$.  Moreover,
	from the work of A. Weinstein and V. Guillemin \cite{W, G},
	there exists a sequence of compactly supported distributions $\beta_i, \ i=0,1,\ldots$ on the real line such that,
	as $k\to\infty$,
	\begin{equation}\label{eq:bandAsym}
		\forall \varphi\in C^\infty(\bbR)\qquad
		\frac{1}{d_k}\sum_{j=1}^{d_k} \varphi (k\,\mu_{k,j}) \sim \sum_{i=0}^{\infty} k^{-i}\,\beta_i(\varphi).
	\end{equation}
	This will be explained in detail in Subsection \ref{avemetod}.
	\begin{definition}
		The distributions $\beta_i$ will be referred to as the {\em band invariants} of the potential $q$.
	\end{definition}
	
	The purpuse  of this paper is to compute the first three invariants $\beta_i$, $i=0,1,2$. We stop at $\beta_2$ because the computations quickly become very complicated. 
 Our calculations will use the symbol calculus developed in \cite{Uribe} for pseudodifferential operators on the $n$-sphere that commute with the spherical Laplacian.
 This calculus, in turn, is based on the asymptotic expansion of the Berezin symbol  of such operators (see Definition \ref{def:BSymbol}). 

 \medskip
	In order to state our results we introduce the 
	unit tangent bundle of $\S$,
	\begin{equation}\label{}
		\calZ := \{(\xi,\eta)\in \S\times\S\;|\; \xi\cdot\eta = 0\}\subset T\S,
	\end{equation}
	where the tangent bundle projection $\pi_T:\calZ\to \S$ is projection onto the first factor.  
	Geodesic flow, re-parametrized by arc length (i.e. the Hamilton flow of the Riemannian norm function on
	$T^*\S\setminus\{0\}$), induces a free $S^1 = \bbR/2\pi\bbZ$ action on $\calZ$.  
	We let
	\begin{equation}\label{}
		\calO := \calZ/\bbS^1
	\end{equation}
	be the quotient space, which can also be thought of as 
	the space of oriented great circles in $\S$ (periodic geodesics).
	It is easy to check that the map
	\begin{equation}\label{eq:TotAngMom}
		\calZ\ni (\xi,\eta) \mapsto \xi\times\eta \in\bbR^3
	\end{equation}
	is constant along ${\mathbb S}^1$ orbits (geodesics), and that it induces a diffeomorphism 
	between $\calO$ and a unit sphere.  Therefore $\calO$ is  diffeomorphic to the original $\S$.    It will be
	important, however, to distinguish between $\S$ and $\calO$, so we will keep this notation.   
	From the point of view of (\ref{eq:TotAngMom}), the correspondence between oriented speed-one geodesics $\gamma\subset \S$ and 
	points on the sphere $\calO$ 
	is: to $\gamma$ we associate its total angular momentum vector.
	
	It will be very convenient to identify $\calZ$ with the following subset of $\bbC^3$:
	\begin{equation}\label{eq:IdentifyUnitTgt}
		\calZ \cong \{z\in\bbC^3\;|\; z\cdot z=0\ \text{and}\ |z|^2=2\},\quad \text{by the map}\ 
		(\xi,\eta) \mapsto  z=\xi + i\eta.
	\end{equation}
	We note, for future reference:
	\begin{lemma}\label{lem:GeoFlowCpx}
		Under the previous identification, the time $t$ map of geodesic flow corresponds to multiplication by $e^{it}$.
	\end{lemma}
We endow $\calZ$ with the unique normalized SO(3) invariant measure $dz$.   The quotient map induces a corresponding push-forward measure 
$d[w]$ on the space $\calO$. (Notice that $d[w]$ is normalized as well.)


	
	We denote by  $\Delta_{\calO}$ and $\nabla_{\calO}$  the Laplacian and gradient operators respectively given by the spherical Riemannian structure  of $\calO$.
	%
	Finally, we will need the following Radon transform:
	\begin{equation}\label{}
		\begin{array}{rcl}
			C^\infty(\S) & \longrightarrow & C^\infty(\calO)\\
			f & \mapsto & \hat{f}([z]) := \frac{1}{2\pi}\int_{[z]} f\, ds\label{Radon}
		\end{array}
	\end{equation}
	where $ds$ denotes arc length. Here $[z]=\pi_{\calO}(\xi,\eta)$ with $z=\xi+\imath\eta$ and $\pi_{\calO}: \calZ\to\calO$ the natural projection
    (i.e. the quotient map), and $[z]$
    is being thought of as a great circle in $\S$.
 We will also denote the Radon transform by
	\[
	\calI(f):= \hat{f},
	\]
	which is much more practical when $f$ is given by a long expression.
	
	We can now state the main theorem: 
	\begin{theorem}\label{thm:Main}  For every $\varphi\in C^\infty(\bbR)$
 there exist constants $\beta_\ell(\varphi)\in\bbR$, $\ell=1,2,\ldots$, such that
	\[
\quad \frac{1}{2k+1} \sum_{j=1}^{2k+1} \varphi(k\mu_{k\,j}) \sim 
	\sum_{\ell\geq 0} \beta_\ell(\varphi) k^{-\ell}.
	\]
Moreover,
\begin{equation*}\label{}
	\beta_0(\varphi) =  \int_\calO \varphi( \hat{q}/2)\, d[w],
\end{equation*}

	\begin{equation*}\label{}
	\beta_1(\varphi)= 	\int_\calO \varphi'(\hat{q})\,\left[ \frac{1}{4} \Delta_{\calO}\hat{q}+ q_1\right]\, d[w]
\end{equation*}
where

	\[
	q_1 = \frac 14{\Ical}\left(-3q-\partial_r q+\frac{1}{2}\Delta_{\S} q \right),
	\]
	and 

\begin{equation*}\label{}
	\beta_2(\varphi)= 	\int_\calO \varphi'(\hat{q})\,\Gamma_1 \, d[w] +
	\int_\calO \varphi''(\hat{q})\,\Gamma_2 \, d[w],
\end{equation*}
with
\[
\Gamma_1 = q_2 -\frac 14 \Delta_{\calO} q_1-\frac{7}{96} \Delta_{\calO}^2 \hat{q},
\]
\begin{alignat*}{2}
	\Gamma_2 = &\frac{7}{96}(\Delta_{\calO} \hat{q})^2 + \frac{5}{96} \Delta_{\calO}(\vert\nabla_{\calO} \hat{q}\vert^2) + \frac{1}{4} q_1\Delta_{\calO} \hat{q} \, +\\
&	+ \frac 12\left(q_1^2 + \langle\nabla_{\calO} \hat{q},\nabla_{\calO} q_1\rangle + D_2(\hat{q}, \hat{q})\right).
\end{alignat*}
$D_2$, given in equation \eqref{eq:Ddos},  is a bilinear second order differential operator and

\[
	 q_2 =  \frac 18 \calI \left(\frac{307}{32}q+2q^2+5\partial_rq+\partial^2_rq-\frac{9}{8}\Delta_{\S}q+\frac{1}{8}\Delta^2_{\S}q-\frac{1}{2}\partial_r\Delta_{\S}q\right)+W
     \]
	where $W:\calO\rightarrow\C$ is the function given  by
\[
		W([z])=\frac{-1}{32\pi^2}\int_0^{2\pi}t\int_0^{2\pi}\lbrace\phi^*_{t+s}(q/\vert\xi\vert),\phi^*_s(q/\vert\xi\vert)\rbrace(z)ds\,dt,
\]
$\phi_t$ being the geodesic flow.

		Here we have restricted the functions appearing on the right-hand sides to $\S=\partial\B$ before 
		taking their Radon transforms.
	\end{theorem}

	In case the restriction of $q$ to $\S$ is an odd function, then $\hat{q}$ is identically zero.  
	In that case, to obtain a meaningful theorem one needs to rescale the $\mu_{k,j}$ by a factor of $k^2$:

 \begin{theorem}\label{thm:casoImpar}
     If the restriction of $q$ to $\S$ is an odd function, then the spectral clusters of $\Lambda_q$ are of size $O(k^{-2})$,
\[ 
|\mu_{jk}|\leq \frac{C}{k^2},
\] 
and for all $\varphi\in C^\infty(\bbR)$,
	\[
	\quad \frac{1}{2k+1} \sum_{j=1}^{2k+1} \varphi({k^2}\mu_{k\,j}) \sim 
	\sum_{\ell\geq 0} \widetilde\beta_\ell(\varphi) k^{-\ell},
	\]
where
	\[
	\widetilde\beta_0(\varphi) = -\frac 14 \int_\calO \varphi\left(\calI(\partial_r q)\right)\, d[w]
	\]
 and
 	\[
	\widetilde\beta_1(\varphi) =
 \int_\calO \varphi'(-\calI(\partial_r q)/4 )\left( -\frac{1}{16} \Delta_\calO(\calI(\partial_r q)) + \tilde q_1 \right)\, \, d[w],
 \]
 where $\tilde q_1$ is given by (\ref{eq:qoneOdd}).
 
 \end{theorem}
	
	\subsection{Outline of the proof}\label{avemetod}  The computation of the $\beta_i$ combines three sets of ideas.
	
	\subsubsection{The averaging method}\label{averagingmethod}
Given $T$ a linear operator defined on $\S$, we define the averaged operator by
\begin{equation}
    T^\ave := \frac{1}{2\pi} \int_0^{2\pi} e^{it\Lambda_0}\, T \, e^{-it\Lambda_0}\, dt, 
\end{equation}
We remark that $T^\ave$ commutes with the Laplacian on $\S$ (and therefore with $\Lambda_0$), and has the property that $\Pi_k T^\ave \Pi_k  = \Pi_k T \Pi_k$. 

	Following the work of A. Weinstein \cite{W},  V. Guillemin proved (\cite{G},  Lemma 1, Section 1) that one can conjugate
	$\Lambda_{q}$ to an operator of the form
	\begin{equation}\label{eq:TheOperatorQ}
		\Lambda_{q}^\# = \Lambda_0 + Q, \qquad\text{where}\quad [Q, \Delta_{\S}] = 0
	\end{equation}
	and $Q$ is a pseudodifferential operator on $\S$ of order $(-1)$ with principal symbol, when restricted to $\calZ$, equal to  the Radon transform 
	of the restriction of $q/2$  to the boundary $\S$.   (See also Colin de Verdi\`ere \cite{CdV} for an alternative approach to
	eigenvalue cluster asymptotics.)
	The operator $Q$ is equal to the average of $S$,
	\begin{equation}\label{eq:Save}
		S^\ave = \frac{1}{2\pi} \int_0^{2\pi} e^{it\Lambda_0}\, S\, e^{-it\Lambda_0}\, dt, 
	\end{equation}
	plus an operator of order $(-3)$ whose principal symbol we will compute in Section \ref{seccionPromedio}.

An important consequence of \eqref{eq:TheOperatorQ} is the following: given 
$k\in\N$, consider the restriction of $Q$ to the space $\calH_k$.  Since $Q$ commutes with $\Lambda_0$ then its restriction  leaves  $\calH_k$ invariant.  Let $\nu_{k,j}$, $j=1,\ldots,d_k$ be the eigenvalues of $Q|_{\calH_k}$. Thus $\lbrace k+\nu_{k,j} | j=1,\ldots,d_k \rbrace$ is a subset of $d_k$ eigenvalues of $\Lambda_{q}^\#$. Since $Q$ is a pseudodifferential operator of order $(-1)$ then, as in the proof of Theorem \ref{teoremaclusters}, one can show that $\nu_{k,j} = O(1/k)$. Since $\Lambda_{q}$ and 
$\Lambda_{q}^\#$ have the same spectrum then,  
for $k$ sufficiently large, we know that the spectrum of $\Lambda_{q}^\#$ is the set $\lbrace k+\mu_{k,j} | j=1,\ldots,d_k \rbrace$. Therefore, after reordering,  we
can assume that
$\nu_{k,j}=\mu_{k,j}$, $j=1,\ldots,d_k$.
Hence
\begin{equation}\label{eq:aveTrace}
\forall \varphi\in C^\infty_0(\bbR),\qquad
			\frac{1}{d_k}\sum_{j=1}^{d_k} \varphi (k\,\mu_{k,j}) 
   = \frac{1}{d_k}\, \tr\left(\varphi(\Lambda_0 Q)|_{\calH_k} \right).
 \end{equation}


	\subsubsection{The Berezin symbol calculus}
	The above leads to the consideration of 
	the ring of pseudodifferential operators
	that preserve the decomposition (\ref{eq:sphHarm}):

	\begin{definition}(\cite{Uribe})\label{def:BSymbolU}
 We will denote by $\frakR$ the ring of pseudodifferential operators on $\S$ that commute with $\DS$.  
	\end{definition}

Operators in $\frakR$ have a Berezin symbol that is defined in terms of
a family of coherent states that we now introduce.
 	To each $z\in\calZ$ regarded as a complex vector $z = \xi+i\eta\in\bbC^3$, we associate the function
	\begin{equation}\label{}
		\alpha_z: \S\to\bbC,\qquad \alpha_z(x) := x\cdot z = x\cdot\xi + i x\cdot\eta.
	\end{equation}
	It is known that, for any $k\in\bbN$ and any $z$ as above,
	\begin{equation}\label{}
		\alpha_z^k\in\calH_k,
	\end{equation}
	and it is clear that
	\begin{equation}\label{eq:PhaseCS}
		\alpha_{e^{it }z}^k = e^{itk} \alpha_z^k.
	\end{equation}
Given $k\in\N$,  we will refer to the function $\alpha_{z}^k$ as the coherent state in $\calH_k$ generated by $\alpha$. 

\medskip
Using Schur's lemma and the SO(3) irreducibility of the spaces $\calH_k$ , one can show that the  orthogonal projector $\Pi_k:L^2(\S)\rightarrow \calH_k$ can be written in terms of coherent states as
\begin{equation}\label{procost}
    \Pi_k \Psi = \frac{d_k}{\|\alpha_z^k\|^2_{L^2(\S)}} \int_{\calZ} \inner{\Psi}{\alpha_z^k} \alpha_z^k dz,  \;\;  \Psi\in L^2(\S).
\end{equation}
One also has that for a linear operator $T$ on $\S$ and every $k\in\bbN$
\begin{equation}\label{caltrazas}
    \tr (\Pi_k T \Pi_k) =  \frac{d_k}{\|\alpha_z^k\|^2_{L^2(\S)}} \int_{\calZ} \inner{T\alpha_z^k}{\alpha_z^k} dz.
\end{equation}

  \begin{definition}(\cite{B})\label{def:BSymbol}
Given a linear operator $T$ on $L^2(\S)$ whose domain contains the functions $\alpha_z^k$, we define  
its Berezin (or covariant) symbol as  the function

		\[
		\frakS_T: \calO\times\bbN\to \bbC
		\]
		given by
		\begin{equation}\label{eq:defBerSymb}
			\frakS_T([z], k) := \frac{\inner{T(\alpha_z^k)}{\alpha_z^k}}{\inner{\alpha_z^k}{\alpha_z^k}},
		\end{equation}
		where $z=\xi+i\eta$ with $(\xi,\eta)\in\calZ$, $[z]\in\calO$ is the projection of $(\xi,\eta)$ to the set $\calO$,
		and the inner products are in $L^2(\S)$. Note that $\frakS_T = \frakS_{T^\ave}$.
  \end{definition}
  
	\begin{remark}
		By (\ref{eq:PhaseCS}),  the right-hand side of (\ref{eq:defBerSymb}) depends only on $[z]$, the orbit $S^1$ of $z$.
	\end{remark}

It turns out that given $A\in\frakR,$ there  exists an asymptotic expansion $\frakS_A \sim k^d\sum_{j=0}^\infty a_j\, k^{-j}$ as $k\to \infty$, where 
	$a_j\in C^\infty(\calO)$ for all $j$, and $d$ is the
	 order of $A$.  The symbol calculus alluded to above, (\cite{Uribe}),
	gives the  expansion of the Berezin symbol of a composition in terms of the expansions of the Berezin symbols of the 
	factors. This is explained in some detail in Section 4 and Appendix A.
	
	On the other hand, from \eqref{caltrazas} we have that  for all $A\in\frakR$ of order $d$
	\begin{equation}\label{eq:traceProperty}
		\frac{1}{d_k}\tr\left(A|_{\calH_k}\right) = \int_\calO \frakS_A(\cdot , k) d[\omega]
		\sim  k^d\sum_{j=0}^\infty k^{-j}\, \int_\calO a_j\, d[\omega].
	\end{equation}
	From (\ref{eq:aveTrace}) and  \eqref{eq:traceProperty} we see that, in order to compute the
 band invariants, we need to know the asymptotic expansion of $\frakS_{\varphi(\Lambda_0Q)}.$ To do that we use the following functional calculus
 formula
	\begin{equation}\label{eq:calcSymb}
		\varphi(\Lambda_0Q)=\frac{1}{\sqrt{2\pi}}\int_{-\infty}^{\infty}\exp (it\Lambda_0
		Q)\mathcal{F}^{-1}(\varphi)(t)dt, \end{equation} where $\mathcal{F}$ is the Fourier transform, in order 
	to reduce our problem to finding the asymptotic expansion of $\frakS_{\exp (it\Lambda_0
		Q)}.$
  In section \ref{seccuatro} we find the first few terms of this expansion in
  terms of that of
  $\frakS_Q$. Finally, using the averaging method, 
  the computation 
  of the first few terms in the expansion of $\frakS_Q$
  is reduced to finding the first few terms of the asymptotic expansion for the Berezin symbol of the operator $\Lambda_q$ itself.
	
	\subsubsection{Computation of the Berezin symbol of $\Lambda_q$}
 In Section \ref{sec calc_Berezin}, we will  compute the first three terms of the expansion of the Berezin symbol of $\Lambda_{q}$, which will involve
the Radon transform of certain compositions of powers of normal derivatives of $q$ on the sphere and of its spherical Laplacian. 
We remark that the study of the Berezin symbol of a given operator is of intrinsic
interest (for example, the case of Toeplitz operators in Bergman and Bargmann spaces).

	\medskip
	The rest of the paper is organized as follows.  
	In Section \ref{seccionPromedio} we recall the averaging method, highlighting some details that we need.
	In Section \ref{seccuatro} we summarize
	the symbol calculus of \cite{Uribe} adapted to the present situation, 
	and we conclude our calculations in Section \ref{seccionFinal}. 
    
 We provide three appendices.  For the interested reader, in Appendix A we explain 
 how the symbol calculus for the ring $\frakR$ is the same as the covariant
 symbol calculus
 of Berezin-Toeplitz (B-T) operators on the K\"ahler manifold $\calO$
 (which has a natural K\"ahler structure). 
The key ingredient is the 
relationship between pseudodifferential operators on the sphere and Toeplitz operators defined on a suitable Hardy space on the set $\calZ$, following 
work of Guillemin, \cite{Gu1984}. 

The identification of the symbol calculus of $\frakR$ with a Berezin-Toeplitz calculus is new.

 Appendix B is devoted to some details of the computations for 
 section \ref{sec calc_Berezin}, and Appendix C to
 a proof of Theorem \ref{teoremaclusters}, which we include for completeness.


\section{The Berezin symbol of $\Lambda_q$}\label{sec calc_Berezin}
\label{seccionBerezin}
The aim of this section is to find the first terms  in $k$ of  the asymptotic expansion of the  Berezin symbol of $\Lambda_q$.
Since $\Vert \alpha_z^k\Vert_{L^2(\B)}^2=\frac{\pi}{k+1} B(k+2,1/2))$ has a well known asymptotic expansion, where  $B(\cdot,\cdot)$ is the Beta function, then we reduce the problem to find the first few terms of the matrix elements $\langle \Lambda_q(\alpha_z^k),\alpha_z^k\rangle_{L^2(\S)}$.
We will accomplish this in the following steps:
first, we will show that we can write
\begin{equation*}
	\langle \Lambda_q(\alpha_z^k),\alpha_z^k\rangle_{L^2(\S)}=k\Vert \alpha_z^k\Vert_{L^2(\S)}^2+\langle q\alpha_z^k,\alpha_z^k\rangle_{L^2(\B)}+\langle \Rcal_q(-q\alpha_z^k),q\alpha_z^k\rangle_{L^2(\B)}
\end{equation*}  
\begin{equation*}
	=:I_1+I_2+I_3,
\end{equation*}
where $\Rcal_q$ is as in \eqref{defresolvente}.

The expansion of $I_2$ will be obtained by the stationary phase method. Next,  $I_3$ would require in principle the Green's function for $-\Delta + q$. To avoid
this difficulty we   prove that a  Neumann-type expansion holds,
\begin{equation}\label{expansion asint}
	\langle \Rcal_q(-q\alpha_z^k),q\alpha_z^k\rangle_{L^2(\B)}\sim\sum_{j=1}^\infty\langle(\Rcal_0\circ M_{-q})^j (\alpha_z^k),q\alpha_z^k\rangle_{L^2(\B)}, 
\end{equation}
where  $M_{-q}$ is the multiplication operator by $-q$. 
The final step is  to expand $\sum_{j=1}^\infty\langle(\Rcal_0\circ M_{-q})^j (\alpha_z^k),q\alpha_z^k\rangle_{L^2(\B)}$ for a few values of $j$ using an integration by parts argument and obtaining  terms similar to $I_2$. 

The main result of this section is the following.
\begin{theorem}\label{asintotica_berezin}For any $z\in \calZ$ we have
	\begin{equation*}
		\frak{S}_{\Lambda_q}(z,k)=k+\frak{S}_{S}(z,k)
	\end{equation*}
	where 
	\begin{equation*}
		\frak{S}_{S}(z,k)=\frac{\mathcal{I}(q)(z)}{2k}+\frac{\mathcal{I}(-3q-\partial_r q+\Delta_{\S} q)(z)}{(2k)^2}
	\end{equation*}
	\begin{equation*}
		+\frac{\mathcal{I}(\frac{307}{32}q+2q^2+5\partial_rq+\partial^2_rq-\frac{9}{8}\Delta_{\S}q+\frac{1}{8}\Delta^2_{\S}q-\frac{1}{2}\partial_r\Delta_{\S}q)(z)}{(2k)^3}
	\end{equation*}

\begin{equation}\label{expasint} 
		+\frac{\mathcal{I}(C)(z)}{(2k)^4}+o\left(\frac{1}{k^4}\right),
	\end{equation}
	where  $C$ is is a linear combination of terms of the form $\partial_r^\ell \Delta_{\S} ^m $ with $\ell+m\leq 3$ and $\partial_r^\ell \Delta_{\S} ^m q^2 $ with $\ell+m\leq 2.$ 
	uniformly on  $\calZ$.
\end{theorem}

\begin{proposition}\label{elemento de matriz} For any $z\in\calZ$,
	\begin{equation*}
		\langle \Lambda_q(\alpha_z^k),\alpha_z^k\rangle_{L^2(\S)}=\sqrt{\frac{\pi}{k}}\left(\frac{\mathcal{I}(q)(z)}{2k}+\frac{\mathcal{I}(A_2(q))(z)}{(2k)^2}+\frac{\mathcal{I}(q-q^2)(z)}{(2k)^3}\right.
	\end{equation*}
	\begin{equation*}
		+\left.\frac{\mathcal{I}(A_4(q)-Aq^2-B\partial_r q^2+C\Delta_{\S} q^2)(z)}{(2k)^4} +O(1/k^5)\right),
	\end{equation*}
	uniformly in $\calZ$,	where each $A_j(q)$ is a linear combination of terms of the form $\partial_r^\ell \Delta_{\S} ^m q$ and $\partial_r^\ell \Delta_{\S} ^m q^2 $respectively with $\ell+m\leq j-1$ and uniformly in $\calZ$. We have in particular
	\begin{align*}
		A_2(q)=&-\frac{15}{4}  q- \partial_r q+\frac{1}{2} \Delta_Sq ,\\
		\\
		A_3(q)=&\frac{405}{32} q+\frac{23}{4}\partial_r q+ \partial^2_r q -\frac{3}{2} \Delta_Sq+\frac{1}{8}\Delta_S^2q \\
		&-\frac{1}{2} \Delta_S \partial_rq { - q^2}.
	\end{align*}
\end{proposition}

To justify the asymptotic expansion \eqref{expansion asint}, we need the following technical lemma.  

\begin{lemma}\label{decaimientoBerezin}
	\begin{itemize} If $p,q\in C^\infty (\overline{\B})$
		\item[a) ]$\langle (\Rcal_0\circ M_{-q})^j(\cohe),p\alpha_z^k\rangle_{L^2(\B)}=O(k^{-2(j-1)-2-3/4}), j\geq 1.$

		\item[b)] If $u= \Rcal_q(-q\alpha_z^k),$  then $\langle (\Rcal_0\circ M_{-q})^j(u),q\ecohe  \rangle_{L^2(\B)}=O(k^{-2j-2-3/4}).$

		\item[c)]
		$\langle (1-r)p(\Rcal_0\circ M_{-q})^j(\cohe),\alpha_z^k\rangle_{L^2(\B)}=O(k^{-2(j-1)-3-3/4}).$
	\end{itemize}
\end{lemma}
\begin{proof}

	\textit{a)}  Given $j\geq 1$, write $v_k=p(r,\theta,\varphi)(\Rcal_0\circ M_{-q})^j (\ecohe)$. Then
	\begin{equation*}
		\langle (\Rcal_0\circ M_{-q})^j(\ecohe),p\alpha_z^k\rangle_{L^2(\B)}=\int_0^1\int_0^{2\pi}\int_0^\pi v_k(r,\theta,\varphi )r^{k+2}\sin (\varphi)^{k+1}e^{ik\theta}d\varphi d\theta dr
	\end{equation*}
	\begin{equation}\label{fcuadratica}
		=\int_0^1r^{k+2}\int_0^\pi \hat{v_k}(r,\cdot,\varphi)(-k) \sin (\varphi)^{k+1}d\varphi dr,
	\end{equation}
	where $\hat{v_k}(r,\cdot,\varphi)(-k) $ is the $(-k)$-th Fourier coefficient of $v_k(r,\cdot,\varphi)$ (not to be confused with the notation for the  Radon transform introduced in \eqref{Radon}.
	We have (see Evans \cite[Ch. 6,.3 Th. 5]{E}) that $\Rcal_0: H^m(\B)\rightarrow H^{m+2}(\B)\cap H_0^{1}(\B)$ is bounded. Iterating this result to calculate $(\Rcal_0\circ M_{-q})^j(\ecohe)$,  we obtain 
	from \eqref{normas de coherentes},
	\begin{equation*}
		\Vert v_k \Vert_{H^{2j}(\B)} =O(k^{-\frac{3}{4}}),
	\end{equation*}
	(the constant may depend on $j$).
	Next,  we use  the Sobolev embedding theorem (\cite[Ch.5.6,Th. 6]{E})
	\begin{equation}
		H^m(\B)\subset C^{m-2,1/2}(\overline{\B}) 
	\end{equation} for any non-negative integer $m$, where $C^{m,\gamma}(\overline{\B})$ denotes the space of functions in $\overline{\B}$ with  H\"{o}lder continuous  derivatives of order $m$ and exponent $0\leq \gamma<1$.
	In particular, $$\Vert v_k\Vert_{C^{2(j-1),1/2}(\overline{\B}) }\leq Ck^{-\frac{3}{4}}.$$
	Hence  the H\"{o}lder norm in the circle
	
	\begin{equation*}
		\Vert v_k(r,\cdot,\varphi)\Vert_{C^{2(j-1),1/2}(\mathbb{S}^ 1)}=O(k^{-\frac{3}{4}}).
	\end{equation*}
	Then (see Katznelson \cite{K} p. 22) we have the estimate of the \textit{n-th} Fourier coefficients of $v_k(r,\cdot,\varphi)$
	\begin{equation*}
		\vert n^{2(j-1)+1/2}\hat{v_k}(r,\cdot,\varphi)(n)\vert\leq Ck^{-\frac{3}{4}},
	\end{equation*}
	for all $n\in\Z$.  Letting $n=-k$, 
	\begin{equation}
		\vert \hat{v_k}(r,\cdot,\varphi)(-k)\vert\leq \frac{C}{k^{2(j-1)+1/2+3/4}}.
	\end{equation}
	Finally by \eqref{fcuadratica} we obtain 
	\begin{equation}
		\vert\langle (\Rcal_0\circ M_{-q})^j(\ecohe),p\alpha_z^k\rangle_{L^2(\B)}\vert\leq \frac{C}{k^{2(j-1)+1+1/2+1/2+3/4}}=\frac{C}{k^{2(j-1) +2+3/4}}.
	\end{equation}
	
	\textit{b}) The proof is the same as for \textit{a}) except that since  $u= \Rcal_q(-q\alpha_z^k)$, we start with $\Vert u\Vert_{H^2(\B)}\leq Ck^{-3/4}$ . Then $ \Vert q(\Rcal_0\circ M_{-q})^ju \Vert_{H^{2(j+1)}(\B)} =O(k^{-\frac{3}{4}})$ and the proof follows as before.
	
	\textit{c}) The proof is a variant  of \textit{a}) replacing \eqref{fcuadratica} by
	\begin{equation*}
		\int_0^1r^{k+2}(1-r)\int_0^\pi \hat{v_k}(r,\cdot,\varphi)(k) \sin (\varphi)^{k+1}d\varphi dr,
	\end{equation*}
	and $v_k=p(r,\theta,\varphi)(\Rcal_0\circ M_{-q})^j (\ecohe).$
\end{proof}

We have the following preliminary expansion of $ \langle \Lambda_q(\alpha_z^k),\alpha_z^k\rangle_{L^2(\S)}$,  which is the  starting point to prove Proposition \ref{elemento de matriz}.

\begin{proposition}\label{primera asintotica} For any $z\in\calZ$,
	\begin{equation*}
		\langle \Lambda_q(\alpha_z^k),\alpha_z^k\rangle_{L^2(\S)}=k\Vert\alpha_z^k\Vert^2_{L^2(\S)}+\sum_{j=0}^N T_j(z,k) +R_N(z,k),
	\end{equation*}
	with $T_j(z,k)=\langle (\Rcal_0\circ M_{-q})^j(\ecohe),q\alpha_z^k\rangle_{L^2(\B)}$ and the residues  $R_N(z,k)= O(k^{-2(N+1)-3/4})$ uniformly in $\calZ$.
\end{proposition}

\begin{proof}

	First, notice that since $\alpha_z^k$ is harmonic,  then the solution of  \eqref{Sch}-\eqref{BC} with $f=\alpha_z^k$ can be written as 
	\begin{equation*}
		u=\alpha_z^k+v,
	\end{equation*}
	where $v$ is a solution of \eqref{ISE} for $F=-\alpha_z^k q$, namely $v=\Rcal_q(-\alpha_z^k q)$ and
	\begin{equation}
		u=\alpha_z^k +\Rcal_q(-\alpha_z^k q).
	\end{equation}
	Thus
	\begin{equation*}
		\Lambda_q (\alpha_z^k)=\frac{\partial u}{\partial n}=k\alpha_z^k+\frac{\partial v}{\partial n}.
	\end{equation*}
	Hence,  by  Green's formula, considering that $\alpha_z^k$ is harmonic and that   $\Rcal_q(q\alpha_z^k)= 0$ on $\S$ we have
	\begin{align*}
		\langle \Lambda_q(\alpha_z^k),\alpha_z^k\rangle_{L^2(\S)}&=k\langle \alpha_z^k,\alpha_z^k\rangle_{L^2(\S)}- \langle \frac{ \partial}{\partial n}\Rcal_q(\alpha_z^k q),\alpha_z^k\rangle_{L^2(\S)}\\
		&=k\langle \alpha_z^k,\alpha_z^k\rangle_{L^2(\S)} -\int_\B \Delta\left(\Rcal_q(q\alpha_z^k)\right)\overline{\alpha_z^k}dx\end{align*}
	\begin{equation}\label{tres partes}
		=k\Vert \alpha_z^k\Vert_{L^2(\S)}^2+\langle q\alpha_z^k,\alpha_z^k\rangle_{L^2(\B)}+\langle \Rcal_q(-q\alpha_z^k),q\alpha_z^k\rangle_{L^2(\B)}
	\end{equation}    
	Next,  if $u= \Rcal_q(-q\alpha_z^k)$ then 
	\begin{equation}\label{funcional}
		u=\Rcal_0 (-q\alpha_z^k)+\Rcal_0 (-qu).
	\end{equation}
	Moreover,  iterating \eqref{funcional}  we have for any $N\geq 1$
	\begin{equation}
		u= \sum_{j=1}^N (\Rcal_0\circ M_{-q})^j (\alpha_z^k)+ (\Rcal_0\circ M_{-q})^N (u).
	\end{equation}
	Hence 
	\begin{equation}\label{recursiva}
		\langle \Rcal_q(-q\alpha_z^k),q\alpha_z^k\rangle_{L^2(\B)}=\sum_{j=1}^N\langle(\Rcal_0\circ M_{-q})^j (\alpha_z^k),q\alpha_z^k\rangle_{L^2(\B)}+\langle (\Rcal_0\circ M_{-q})^N (u),q\alpha_z^k\rangle_{L^2(\B)}.
	\end{equation}
	Finally, using   Lemma \ref{decaimientoBerezin}  in the expansion \eqref{recursiva}  we have that 
	\begin{equation*}
		R_N(z)=\langle (\Rcal_0\circ M_{-q})^N (u),q\alpha_z^k\rangle_{L^2(\B)}=O(k^{-2(N+1)-3/4}).
	\end{equation*}
\end{proof}


Before starting the proof of Proposition  \ref{elemento de matriz},  note that it is enough to prove  it for a particular element $z\in\calZ$.
In fact, as noticed in \eqref{elemmatT}, if $T\in O(n)$  
\begin{equation*}
	\langle \Lambda_{q}f,f\rangle_{L^2(\S)}=\langle \Lambda_{q\circ T}(f\circ T),f\circ T\rangle_{L^2(\S)}.
\end{equation*}
Also 
\begin{equation}
	\alpha_z^k\circ T=\alpha^k_{T^{-1}z},
\end{equation}
with $T^{-1}z=T^{-1}\xi+iT^{-1}\eta\in\calZ$.
Hence
\begin{equation}\label{ber_con_T}
	\langle \Lambda_q(\alpha_z^k),\alpha_z^k\rangle_{L^2(\S)}=\langle \Lambda_{q\circ T}\alpha_{T^{-1}z}^k,\alpha_{T^{-1}z}^k\rangle_{L^2(\S)}.
\end{equation}
Let $z_0\in\calZ$ and suppose that the asymptotic expansion of Theorem \ref{elemento de matriz} holds for $z_0$ and any potential $q$.  If $z\in \calZ$ there exists $T\in O(n)$ such that $z_0=T^{-1}z$.  Then by \eqref{ber_con_T} and considering that $ \partial_r (q\circ T)=\partial_r q \circ T$, $ \Delta_{\S} (q\circ T)=\Delta_{\S} q \circ T$ and $\widehat{g\circ T}(z_0)=\hat{g}(z)$,  for any function $g$ on $\S$,  we conclude that the asymptotic expansion of  $\langle \Lambda_q(\alpha_z^k),\alpha_z^k\rangle_{L^2(\S)}$ is precisely \eqref{expasint}.  

Let $(r,\theta,\varphi)$ be the spherical coordinates in $\R^3$ with $\theta$ the azimutal angle.  We will denote for a function $p$ in $\B$
\begin{equation}
	\widetilde{p}(r,\varphi)=\int_0^{2\pi} p(r,\theta,\varphi)d\theta,
\end{equation}
From now on we will assume that $z=(1,i,0)$.  In this case  
$\widetilde{p}(1,\pi/2)=2\pi\widehat{p}([(1,i,0)])$.

We will need the following result whose straightforward proof is  postponed to the Appendix \ref{apendiceB}.
\begin{lemma}\label{Radon de pot de Delta}
	
	Let $z=(1,i,0)\in\calZ$. Then for any $m\ge 1$, $\widetilde{\partial^{2m}_\varphi q}(1,\pi/2)$ is a linear combination (with coefficients independent of $q$) of $\lbrace\widehat{\Delta_{\S}^j q}(z))\rbrace_{1\leq j\leq m}.$
	In particular
	\begin{equation*}
		\widetilde{\partial_{\varphi}^2q}(1,\pi/2)=2\pi\widehat{\Delta_{\S }q}(z)\text{ and }\,\widetilde{\partial_{\varphi}^4q}(1,\pi/2)= 2\pi\widehat{\Delta_{\S }^2q}(z)+4\pi\widehat{\Delta_{\S }q}(z).
	\end{equation*}
\end{lemma}

Let $T_i(k), \, i=0,1...$ as in Proposition \ref{primera asintotica}.
\begin{lemma}\label{estacionaria}  
	$$T_0(z,k)=\langle q\ecohe, \ecohe\rangle_{L^2(\B)}\sim 2\pi\sqrt{\frac{\pi}{k}}\left( \frac{\widehat{q}(z)}{2k}+\sum_{j=2}^\infty \frac{\widehat{A_j(q)}(z)}{(2k)^j}\right),$$
	where each $A_j(q)$ is a linear combination of terms of the form $\partial_r^\ell \Delta_{\S} ^m q$  with $\ell+m\leq j-1,$  and $\widehat{A_j(q)}(z)$ is bounded on  $\calZ$ for each $j$. We have in particular
	\begin{align*}
		A_2(q)=&-\frac{15}{4}  q- \partial_r q+\frac{1}{2} \Delta_Sq ,\\
		\\
		A_3(q)=&\frac{405}{32} q+\frac{23}{4}\partial_r q+ \partial^2_r q -\frac{7}{2} \Delta_Sq+\frac{1}{8}\Delta_S^2q \\
		&-\frac{1}{2} \Delta_S \partial_rq.
	\end{align*}
\end{lemma}
\begin{proof}
	We prove it for  $z=(1,i,0)\in\calZ$. Write $\cohe(x)= r^k\sin^k\varphi e^{ik\theta },$ so that 
	\begin{equation*}
		T_0(z,k)=\langle q\alpha_z^k,\alpha_z^k\rangle_{L^2(\B)}= \int_0^1\int_0^{\pi }\tilde{q}(r,\varphi)r^{2k+2}\sin^{2k+1}\varphi d\varphi dr
	\end{equation*}
	\begin{equation}
		= \int_0^1  \mathcal{J}_k(r)r^{2k+2}dr,
	\end{equation}
	where $\mathcal{J}_k(r)=\int_0^{\pi}\tilde{q}(r,\varphi)\sin^{2k+1}(\varphi) d\varphi=\int_0^{\pi}\tilde{q}(r,\varphi)\sin\varphi e^{2k\Phi(\varphi)i}\varphi d\varphi$,  with $\Phi(\varphi)=-i\log(\sin\varphi).$
	$\Phi$ has a unique critical point at $\pi/2$, and  the stationary phase method yields  the asymptotic expansion
	\begin{equation}
		\mathcal{J}_k(r)\sim \sqrt{\frac{\pi}{k}}\sum_{j=0}^\infty \frac{L_j(r)}{(2k)^j} 
	\end{equation} 
	where
	\begin{equation}\label{Lj}
		L_j(r)=\sum_{m-n=j}\sum_{2m\geq 3n}i^{-j}2^{-m}\left(i\frac{\partial^2}{\partial\varphi^2}\right)^m\left[ \frac{g^n(\varphi)\tilde{q}(r,\varphi)\sin\varphi}{m!n!}\right]_{\varphi=\pi/2}, 
	\end{equation}
	with\begin{equation*}
		g(\varphi)=\Phi(\varphi)-\frac{i}{2}(\varphi-\pi/2)^2.
	\end{equation*}
	
	A routine proof calculation shows that for any $n\in\N$
	
	\begin{itemize}
		\item[\emph{a})] $\partial^ig^n(\pi/2)=0$ for every odd positive integer $i$, 
		\item[\emph{b})] $\partial^ig^n(\pi/2)=0$ for every $i<2n$.
	\end{itemize}
	Now we conclude that 
	\begin{equation*}
		L_j(r)=\sum_{\ell=0}^j a_{j,\ell}\partial_\varphi^{2\ell}\tilde{q}(r,\pi/2).
	\end{equation*}
	
	In fact, any derivative  $\partial_\varphi^{2m}[g^n(\varphi)\tilde{q}(r,\varphi)\sin\varphi]$ is the sum of terms of the form $\partial_\varphi^ig^n(\varphi)\,\partial_\varphi^\ell\tilde{q}(r,\varphi)\,\partial_\varphi^k \sin\varphi$, $i+\ell+k=2m$.
	Then \emph{a}) and \emph{b}) above force that the only nonzero terms appearing in
	\eqref{Lj} are multiples of $\partial_\varphi^{2\ell}\tilde{q}(r,\pi/2)$, $\ell\leq j$.
	In particular we have 
	\begin{itemize}
		\item $L_0(r)=\tilde{q}(r,\pi/2)$,
		\item $L_1(r)=-\frac{3}{4}\tilde{q}(r,\pi/2)+\widetilde{\partial_\varphi^2{q}}(r,\pi/2)$,
		\item $L_2(r)=\frac{45}{32}\tilde{q}(r,\pi/2)-2\widetilde{\partial_\varphi^2 q}(r,\pi/2)+\frac{1}{8}\widetilde{\partial_\varphi^4 q}(r,\pi/2)$.
	\end{itemize}

	Now, since $q\in C^\infty(\overline{\B})$, the asymptotic expansion for $\mathcal{J}_k(r)$ is uniform for $r\in [0,1]$ and for any $N$
	\begin{equation*}
		\mathcal{J}_k(r)r^{2k+2}= \sqrt{\frac{\pi}{k}}\sum_{j=0}^N\frac{L_j(r)r^{2k+2}}{(2k)^j} +O\left(\frac{r^{2k+2}}{k^{1/2+N+1}}\right),
	\end{equation*}
	hence
	\begin{equation}\label{T2preliminar}
		T_0(z,k)=\sqrt{\frac{\pi}{k}}\sum_{j=0}^N\frac{1}{(2k)^j}\int_0^1 L_j(r)r^{2k+2}dr+O(k^{-1/2-N-2}).
	\end{equation}
	Integration by parts shows that there exist a sequence of polynomials $p_s(t)$ of degree $s-1$,  such that
	for any $f\in C^\infty[0,1]$
	\begin{equation}\label{momentos}
		\int_0^1 f(t)t^{2k+2}dt\sim\sum_{s=1}^\infty\frac{p_s(\partial)f(1)}{(2k)^s}   + O(\ell^{-(M+1)}).
	\end{equation}
	\begin{align*}
		= &\frac{f(1)}{2k}-\frac{1}{(2k)^2}(3f(1)+f'(1))\\ &+\frac{1}{(2k)^3}(9f(1)+5f'(1)+f''(1)) +O(1/k^4).
	\end{align*}
	Hence
	\begin{equation*}
		\frac{1}{(2k)^j}\int_0^1 L_j(r)r^{2k+2}dr\sim\sum_{s=1}^\infty \frac{p_s(\partial)L_j(1)}{(2k)^{j+s}}.
	\end{equation*}
	Collecting powers of $k$ of the same degree, we obtain 
	\begin{equation*}
		T_0(k)\sim\sqrt{\frac{\pi}{k}}\sum_{j=1}^\infty\frac{\widetilde{B_j(q)}(1,\pi/2)}{(2k)^j}
	\end{equation*}
	where each $B_j(q)$ is a linear combination of terms of the form $\partial_r^\ell \partial^{2m} q$  with $\ell+m\leq j-1.$ 
	After simple calculations, we explicitly get
	\begin{itemize}
		\item $B_1(q)=q,$
		\item $B_2(q)=-\frac{15}{4}q-\partial_rq+\frac{1}{2}\partial_\varphi^2q,$
		\item $B_3(q)=\frac{405}{32}q+\frac{23}{4}\partial_rq+\frac{1}{8}\partial_\varphi^4q-\frac{1}{2}\partial_r\partial_\varphi^2q +\partial^2_rq-\frac{7}{2}\partial^2_\varphi q.$
	\end{itemize}
	Finally,  the lemma follows applying Lemma \ref{Radon de pot de Delta}. 
\end{proof}
\begin{lemma}\label{super lema}
	Let $F,  p, q\in C^\infty(\overline{\B})$, then
	\begin{itemize}
		\item[a)] \begin{equation*}
			\langle \Rcal_0(qF) ,p\ecohe  \rangle_{L^2(\B)}=  -\frac{1}{4k+6}\langle (1-r^2) pqF ,\ecohe  \rangle_{L^2(\B)} 
		\end{equation*}
		\begin{equation*}
			-\frac{1}{4k+6}\langle (1-r^2) \Rcal_0(qF)\Delta p ,\ecohe  \rangle_{L^2(\B)}
			-\frac{2}{4k+6}\langle (1-r^2)\nabla p\cdot\nabla \Rcal_0(qF) ,\ecohe  \rangle_{L^2(\B)}.
		\end{equation*}
		
		\item[b)]\begin{align*}
			-\frac{2}{4k+6}\langle (1-r^2)\nabla p\cdot\nabla \Rcal_0(qF) ,\ecohe  \rangle_{L^2(\B)}= \frac{2}{4k+6}\langle\mathcal{L}(p)\Rcal_0(qF) ,\ecohe  \rangle_{L^2(\B)}\\
			+\frac{2k}{4k+6}\langle \Rcal_0(qF) (1-r^2) \nabla p\cdot \overline{z}  ,\alpha_z^{k-1}\rangle_{L^2(\B)},
		\end{align*}
		with $ \mathcal{L}=\sum_{i=1}^3\partial_i ((1-r^2)\partial_i)$.
	\end{itemize}
\end{lemma}
\begin{proof}
	\textit{a)} For a function $f$ on $\B$ denote by $\Delta_{\S} f(x)$ the Laplace-Beltrami operator in the sphere acting on $\omega$ for $x=r\omega,\, \omega\in \S  .$
	We have \begin{equation*}
		\Delta_{\S}=r^2\Delta-\partial_r(r^2\partial_r).
	\end{equation*}
	Let $Q=\Rcal_0(qF)$, then writing 
	\begin{equation*}
		\langle \Rcal_0(qF) , p\ecohe  \rangle_{L^2(\B)}=\langle pQ,\ecohe  \rangle_{L^2(\B)}
	\end{equation*}
	\begin{equation*}
		=-\langle pQ,\frac{1}{k(k+1)}\Delta_{\S}\ecohe  \rangle_{L^2(\B)}=-\langle \Delta_{\S}(pQ) ,\frac{1}{k(k+1)}\ecohe  \rangle_{L^2(\B)}
	\end{equation*}
	\begin{equation}\label{laprimera}
		=\frac{1}{k(k+1)}\langle \partial_r(r^2\partial_r(pQ)),\ecohe \rangle_{L^2(\B)}-\frac{1}{k(k+1)}\langle r^2\Delta(pQ),\ecohe \rangle_{L^2(\B)}.
	\end{equation}

	Now,  since $Q=0$ on $\S$,  then integrating by parts twice we have
	\begin{equation}
		\frac{1}{k(k+1)}\langle \partial_r(r^2\partial_r(pQ)),\ecohe \rangle_{L^2(\B)}= \frac{1}{k(k+1)}\int_0^1\int_{\S}\partial_r(r^2\partial_r(pQ))r^{k+2}\overline{\cohe} (\omega)d\sigma(\omega)dr
	\end{equation}
	\begin{equation*}
		=\frac{1}{k(k+1)}\int_{\S}\partial_r(pQ)\overline{\cohe}d\sigma+\frac{(k+2)(k+3)}{k(k+1)}\langle \Rcal_0(qF) , p\ecohe  \rangle_{L^2(\B)}
	\end{equation*}
	(by Green's formulas)
	\begin{equation}\label{parte en r}
		=\frac{1}{k(k+1)}\langle \Delta(pQ),\ecohe \rangle_{L^2(\B)}+\frac{(k+2)(k+3)}{k(k+1)}\langle \Rcal_0(qF) , p\ecohe  \rangle_{L^2(\B)}.
	\end{equation}
	Combining  \eqref{laprimera} an \eqref{parte en r} we obtain 
	\begin{equation*}
		\langle \Rcal_0(qF) ,p\ecohe  \rangle_{L^2(\B)}=-\frac{1}{4k+6}\langle (1-r^2)\Delta(pQ),\ecohe \rangle_{L^2(\B)}
	\end{equation*}
	and the proof of \textit{a)} follows since $\Delta Q=qF$.
	The proof of  \textit{b)} is a direct application of Green's formulas using the fact that $\Rcal_0(qF) =0$ in $\S$.
\end{proof}
\begin{remark}\label{gradientes}
	If $F=\ecohe$,  then according to Lemma \ref{decaimientoBerezin} the expression on Lemma \ref{super lema}b) is $O(k^{-3-3/4}).$ 
\end{remark}
\begin{lemma}\label{T1(k)}
	\begin{equation*}
		T_1(z,k)=-\sqrt{\frac{\pi}{k}}\left[\frac{\widehat{q^2}(z)}{(2k)^3}+\frac{A\widehat{q^2}(z)+B\widehat{\partial_r q^2}(z)+C\widehat{\Delta_{\S} q^2}}{(2k)^4} +O\left(\frac{1}{k^{4+3/4}}\right)\right].
	\end{equation*}
\end{lemma}

\begin{proof}
	
	Again,  it suffices to prove the lemma  for $z=(1,i,0)$.  By Lemma \ref{decaimientoBerezin},
	\begin{equation}\label{pasar a R0}
		T_1(k)=-\langle Q ,q\ecohe  \rangle_{L^2(\B)}
	\end{equation}where $Q=\Rcal_0(q\ecohe)$.
	
	By Lemma \ref{super lema}a),
	\begin{align*}
		\langle Q ,q\ecohe  \rangle_{L^2(\B)}=  -\frac{1}{4k+6}\langle (1-r^2) q^2 \ecohe ,\ecohe  \rangle_{L^2(\B)} \\ 
		-\frac{2}{4k+6}\langle (1-r^2)\nabla q\cdot\nabla Q ,\ecohe  \rangle_{L^2(\B)}-\frac{1}{4k+6}\langle (1-r^2) Q\Delta q ,\ecohe  \rangle_{L^2(\B)}
	\end{align*}
	\begin{equation}\label{con J}
		= J_1+J_2+J_3.
	\end{equation}

	Then using Lemma \ref{estacionaria} replacing the function $q$ by $p=(1-r^2)q^2$ we obtain 
	\begin{equation*}
		J_1= \sqrt{\frac{\pi}{k}}\left( A_1\frac{\widetilde{q^2}(1,\pi/2)}{(2k)^3}+\frac{1}{(2k)^4}( B_1\widetilde{q^2}(1,\pi/2)+C_1 \widetilde{\partial_r q^2}(1,\pi/2)\right.
	\end{equation*}\begin{equation}\label{J1}
		\left. +D_1 \widetilde{\partial^2_\varphi q^2}(1,\pi/2))+O(k^{-4-3/4}\right).
	\end{equation}
	Next, by Lemma \ref{super lema}b)
	\begin{equation*}
		J_2= \frac{2}{4k+6}\langle\mathcal{L}(q)Q ,\ecohe  \rangle_{L^2(\B)} +\frac{2k}{4k+6}\langle Q (1-r^2) \nabla q\cdot z,\alpha_z^{k-1}\rangle_{L^2(\B)}
	\end{equation*}
	
	\begin{equation*}
		=J_{2,1}+J_{2,2}.
	\end{equation*}
	Apply again Lemma \ref{super lema}, and use Remark \ref{gradientes} and Lemma \ref{decaimientoBerezin} to see that
	\begin{equation*}
		J_{2,1}=\sqrt{\frac{\pi}{k}}\left(\frac{A_{2,1}\widetilde{q\mathcal{L}q}(1,\pi/2)}{(2k)^{4}}+O(\frac{1}{k^{4+3/4}})\right).
	\end{equation*}
	\begin{equation}\label{j21}
		=\sqrt{\frac{\pi}{k}}\left(\frac{A_{2,1}\widetilde{\partial_r q^{2}}(1,\pi/2)}{(2k)^{4}}+O(\frac{1}{k^{4+3/4}})\right).
	\end{equation}
	To analyse $J_{2,2}$ let $h\in C^\infty$ be a cut-off function such that $h=1$ in $x_1^2+x_2^2+x_3^2>1/2$  and $h=0$ in  
	$x_1^2+x_2^2+x_3^2<1/4$.  Then again by Lemma \ref{super lema},  Remark \ref{gradientes} and considering the exponential decay in $k$ of $\ecohe$ on  $x_1^2+x_2^2<1/4$,
	\begin{align*}
		J_{2,2}&=\frac{2k}{(4k+6)^2}\langle\frac{(1-r^2)^2(q\nabla q\cdot \overline{z}) h}{\overline{\alpha_z}},\ecohe\rangle_{L^2(\B)}+ O(k^{-4-3/4})\\
		&=\frac{k}{(4k+6)^2}\langle\frac{(1-r^2)^2(\nabla q^2\cdot \overline{z}) h}{\overline{\alpha_z}},\ecohe\rangle_{L^2(\B)}+ O(k^{-4-3/4}).
	\end{align*}
	In spherical coordinates
	\begin{equation*}
		\partial_{x_1}=\cos \theta \sin \varphi\,\partial_r-\frac{\sin \theta}{r\sin \varphi}\partial_{\theta}+\frac{\cos\theta\cos\varphi}{r}\partial_\varphi,
	\end{equation*}
	\begin{equation*}
		\partial_{x_2}=\sin \theta \sin \varphi\,\partial_r+\frac{\cos \theta}{r\sin\varphi}\partial_{\theta}+\frac{\sin\theta\cos\varphi}{r}\partial_\varphi,
	\end{equation*}
	then at $r=1$,  $\f=\pi/2$,
	\begin{equation*}
		\frac{\nabla q^2\cdot \overline{z} }{\overline{\alpha_z}}=\frac{e^{-i\theta}(\sin \f\,\partial_r-\frac{i}{r\sin\varphi}\partial_\theta+\frac{\cos\f}{r}\partial_\varphi)}{e^{-i\theta}r\sin \f}(q^2)=\partial_r q^2-i\partial_\theta q^2,
	\end{equation*}
	so that 
	\begin{equation*}
		\frac{\widetilde{(h\nabla q^2\cdot \overline{z}} )}{\overline{\alpha_z}}(1,\pi/2)=\widetilde{\partial_rq^2}(1,\pi/2).
	\end{equation*}
	
	\begin{equation}\label{j22}
		J_{2,2}=\sqrt{\frac{\pi}{k}}\left[\frac{C\widetilde{\partial_rq^2}(1,\pi/2)}{k^4}+O(k^{-4-3/4})\right].
	\end{equation}
	Finally by Lemma \ref{decaimientoBerezin},  $J_3=O(k^{-4-3/4})$, then the proof is complete after summing \eqref{J1}, \eqref{j21} and \eqref{j22}. 
	
\end{proof}

\begin{remark}
	Notice  that $J_3$ in \eqref{con J} includes a term $\widehat{q\Delta q}(z)$ in the power $k^{-5}$.  For the next powers in $k$ terms like the Radon transform of functions $\mathcal{I}^N(q)$ with $\mathcal{I} f=f\Delta f$ or powers of $\Delta_{\S}$
	or $\partial_r$ of such functions will be appearing. 
	
	It is possible to calculate the asymptotics for $T_j(k), j>1$ by applying $j$ times  Lemma \ref{super lema}. 
\end{remark}
\begin{proof} (Theorem \ref{asintotica_berezin}).
	The proof follows from Proposition \ref{elemento de matriz} and  (see Appendix \ref{apendiceB})
	\begin{equation}\label{cocientes betas}
		\frac{1}{2\pi B(k+1,1/2)} =\sqrt{\frac{k}{\pi}}\frac{1}{2\pi}\left(1+\frac{3}{4(2\pi)}-\frac{1}{4(2k)^2}+\frac{191}{64(2k)^3}+O(1/k^4)\right).
	\end{equation}
	and
	\begin{equation*}
		\frak{S}_{\Lambda_q}(z,k)=\frac{\langle \Lambda_q(\alpha_z^k),\alpha_z^k\rangle_{L^2(\S)}}{2\pi B(k+1,1/2)}.
	\end{equation*}
\end{proof}


\section{Averaging and the Berezin symbol of $Q$}\label{seccionPromedio}

Recall that $\frakR$ denotes the ring of classical pseudodifferential operators on $\S$ 
that commute with $\Delta_{\S}$.  Equivalently, a $\Psi$DO $Q$ belongs to $\frakR$ iff $Q(\calH_k)\subset\calH_k$
for all $\forall k=0,1,\ldots$.
As stated in Subsection \ref{averagingmethod}, our interest in this ring is because one has:
\begin{theorem} (\cite{G},  Lemma 1, Section 1)
	Given $q\in\C^\infty(\overline{\B})$, there exists $Q\in\frakR$ of order $(-1)$, self adjoint, such that
	$\Lambda_{q}$ is unitarily equivalent to $\Lambda_q^\# = \Lambda_0+Q$.  
\end{theorem}
We will need an approximation to $Q$ in order to compute the first three terms of its Berezin symbol, and therefore we review
aspects of the proof of this theorem.
Recall that if $A$ is any classical $\Psi$DO on $\S$, we defined
\[
A^\ave = \frac{1}{2\pi}\, \int_0^{2\pi} e^{it\Lambda_0} A e^{-it\Lambda_0}\, dt.
\]
By Egorov's theorem, $A^\ave$ is a $\Psi$DO of the same order as $A$, and its 
principal symbol is the function
\begin{equation}\label{}
	\sigma_A^\ave:= \frac{1}{2\pi}\, \int_0^{2\pi} \phi_t^*\sigma_A\, dt
\end{equation}
where $\phi_t:T^*\S\setminus\{0\}\to T^*\S\setminus\{0\}$ is the Hamilton flow of $\sigma_{\Lambda_0} = |\xi|$.
Moreover, $[A^\ave, \Lambda_0] = 0$, i.e. $A^\ave \in \frakR$.

The goal of this section is to establish the following:
\begin{proposition} \label{thm:aveMethod}
	For any $q\in C^\infty(\overline\B)$, $\Lambda_q$ is unitarily equivalent to an operator of the form
	\begin{equation}\label{}
		\Lambda_{q}^\# =	\Lambda_0 + Q,
	\end{equation}
	where \begin{equation}\label{Q}
		Q=S^\ave +\frac{1}{2} [F, S]^\ave + R,
	\end{equation} and $F$ is either of the operators 
	\begin{equation}\label{originalF}
		F_1 =\frac{-i}{2\pi}\int_{0}^{2\pi}dt\int_0^t e^{is\Lambda_0} S e^{-is\Lambda_0} \, ds
	\end{equation}
	or
	\begin{equation}\label{anotherF}
		F_2 = \frac{i}{2\pi}\int_0^{2\pi} t\,e^{it\Lambda_0} S e^{-it\Lambda_0} \, dt,
	\end{equation}
	and $R$ is a $\Psi$DO of order $(-5)$.
\end{proposition}
\begin{remark}
	The operator $F$ satisfies the key identity
	\begin{equation}\label{eq:laPapaS}
		[F, \Lambda_0] = S^\ave - S.
	\end{equation}
	Moreover, $F_1 = -2\pi iS^\ave + F_2$.
	
	In fact, 
	\begin{equation*}
		[F, \Lambda_0]=\frac{-1}{2\pi}\int_0^{2\pi}t\frac{d}{dt}\left(e^{it\Lambda_0} S e^{-it\Lambda_0}\right)dt= S^\ave - S
	\end{equation*}
	where we have used integration by parts and  $e^{2\pi i \Lambda_0}= I$, and $I$ is the identity operator.
\end{remark}

For completeness we sketch the proof of the proposition.   
We expand the conjugation
\begin{equation}\label{}
	e^{F}\Lambda_q e^{-F} \sim  \Lambda_q + [F,\Lambda_q] + 
	\frac 12 [F, [F,\Lambda_q]] + \cdots 
\end{equation}
This is an expansion in the sense of pseudodifferential operators.
Since $F$ has order $(-1)$ (the same as $S$), $\ad_F(\cdot) := [F, \cdot]$ lowers the order by two.  
Therefore, the dots have order no greater than $(-5)$ (they involve at least $\ad_F^3$).  

In what follows we'll ignore operators of order $\leq -4$, so let us look at 
\begin{equation}\label{}
	\Lambda_q + [F,\Lambda_q] + 
	\frac 12 [F, [F,\Lambda_q]] = \Lambda_0 + S + [F,\Lambda_0] + [F, S] 
	+ \frac 12 [F, [F,\Lambda_0]] + \frac 12 [F, [F,S]].
\end{equation}
The last term is of order $(-5)$ and we discard it.  By equation (\ref{eq:laPapaS})
$S + [F, \Lambda_0] = S^\ave$.
Hence,

\begin{equation}\label{eq:thereforeCanBe}
	e^{F}\Lambda_q e^{-F}= \Lambda_0 + S^\ave + [F,S] + \frac 12 [F, S^\ave - S] + O(-5) =
	\Lambda_0 + S^\ave +  \frac 12 [F, S^\ave + S] + O(-5).
\end{equation}
We iterate the procedure as follows: replace $\Lambda_q$ by   $e^{F}\Lambda_q e^{-F}$ and $S$ by 
$\tilde{S}=\frac 12 [F, S^\ave + S]$. Then define 

$$\tilde{F}= \frac{i}{2\pi}\int_0^{2\pi}t\,e^{it\Lambda_0}\tilde{S}e^{-it\Lambda_0} \, dt. $$
Therefore $\Lambda_q$ can be conjugated to 
\[
\Lambda_0 + S^\ave +  \frac 12 [F, S^\ave + S]^\ave + O(-5).
\]
where we use the notation $O(-5)$ to denote a $\Psi$DO of order at most $(-5)$.
The proposition then follows from:
\begin{lemma}
	$\displaystyle{[F, S^\ave]^\ave = 0}$.
\end{lemma}
\begin{proof}
	We begin by proving that
	\begin{equation}\label{}
		[S, S^\ave]^\ave = 0
	\end{equation}
	which, incidentally, implies that $[F_1, S]^\ave = [F_2, S]^\ave$.
	Indeed, 
	\[
	[S, S^\ave]^\ave = \frac{1}{2\pi}\int_0^{2\pi}e^{it\Lambda_0}
	[ S, S^\ave] e^{-it\Lambda_0}\, dt = \frac{1}{2\pi}\int_0^{2\pi}
	[e^{it\Lambda_0}Se^{-it\Lambda_0}, S^\ave]\, dt = [S^\ave, S^\ave]=0.
	\]
	Similarly, one can verify that
	\begin{equation}\label{basically}
		\forall t\quad [e^{it\Lambda_0}(S)e^{-it\Lambda_0}, S^\ave]^\ave = 0.
	\end{equation}
	Finaly, notice that
	\begin{equation}\label{}
		[F_2, S^\ave]^\ave = -\frac{1}{2\pi}\int_0^{2\pi}t [e^{it\Lambda_0} S e^{-it\Lambda_0}),S^\ave]^\ave\, dt = 0
	\end{equation}
	since the integrand is zero, by (\ref{basically}).  This proves the lemma, and therefore the proposition.
\end{proof}

Combining the proposition above with Theorem \ref{asintotica_berezin}, we obtain:

\begin{corollary}\label{cor:LaPapa}
	The Berezin symbol of the operator $\Lambda_0 Q$ with $Q$ as in \eqref{Q}, satisfies 
	\[
	\frakS_{\Lambda_0 Q} \sim \sum_{j=0}^\infty q_j\, k^{-j}
	\]
	where:
	\begin{equation}\label{eq:q0}
		q_0 =  \frac 12 \calI(q),
	\end{equation}
	\begin{equation}\label{eq:q1}
		q_1 = \frac 14 \calI\left(-3q-\partial_r q+\frac{1}{2}\Delta_{\S} q \right),
	\end{equation}
	and
	\begin{equation}\label{eq:q2}
		q_2 =  \frac 18 \calI \left(\frac{307}{32}q+2q^2+5\partial_rq+\partial^2_rq-\frac{9}{8}\Delta_{\S} q +\frac{1}{8}\Delta^2_{\S} q -\frac{1}{2}\partial_r\Delta_{\S}q\right)
		+ W,
	\end{equation}
	where $W:\calO\to\bbC$ is the function given  by
	\begin{equation}\label{}
		W([z])=\frac{-1}{32\pi^2}\int_0^{2\pi}t\int_0^{2\pi}\lbrace\phi^*_{t+s}(q/\vert\xi\vert),\phi^*_s(q/\vert\xi\vert)\rbrace(z)ds\,dt.
	\end{equation}
	and where the pull-back of $f$ via $\phi_t$ is given by  $\phi_t^*(f)=f\circ\phi_t$, for any function $f$ defined on $T^*S$.   
	\begin{proof}
		From Proposition \ref{thm:aveMethod} we write
		\begin{equation*}
			\Lambda_0 Q= \Lambda_0 S^\ave + \frac{1}{2}\Lambda_0[F, S]^\ave + \Lambda_0 R.
		\end{equation*}
		
		Hence 
		\begin{align*}
			\frakS_{\Lambda_0 Q}([z],k)=& k \left(\frakS_{S^{\ave}}([z],k)+\frakS_{[F, S]^\ave}([z],k)+\frakS_{R}([z],k)\right)\\ =& k \left(\frakS_{S}([z],k)+\frakS_{\frac{1}{2}[F, S]}([z],k)+\frakS_{R}([z],k)\right).
		\end{align*} 
		The first term in this equation is given in \eqref{expasint}.

		Now, for $k\frakS_{\frac{1}{2}[F, S]}(z,k)$, notice first that $\frac{1}{2}[F, S]$ is a pseudodifferential operator of order $-3$. It is well known (see for example \cite[Thm. 4.2]{Uribe} together with Egorov's theorem) that the principal term in the asymptotic expansion of  $\frakS_{\frac{1}{2}[F, S]}(z,k)$ is $1/k^3$ \textit{times} the Radon transform of the principal symbol $\sigma_{\frac{1}{2}[F, S]}$ of $ \frac{1}{2}[F, S]$. The third term  $\frakS_{R}([z],k)$ is $O(k^{-5})$ and we will not consider it because we are only collecting terms upto  order $k^{-3}$. 
		
		Now we compute the leading term of the asymptotic expansion for $\frakS_{\frac{1}{2}[F, S]}([z],k)$:
		\begin{align*}
			\frakS_{\frac{1}{2}[F, S]}([z],k)=&\frac{1}{\langle\cohe,\cohe\rangle}\langle \cohe, \frac{1}{2}[F, S]\cohe\rangle \\
			=& \frac{-i}{4\pi}\int_0^{2\pi}t\frac{\langle \cohe,[e^{it\Lambda_0} S e^{-it\Lambda_0},S]\cohe\rangle}{\langle\cohe,\cohe\rangle} dt\\
			=&  \frac{-i}{8\pi^2k^3}\int_0^{2\pi}t\int_0^{2\pi}\sigma_{[e^{it\Lambda_0} S e^{-it\Lambda_0},S]}\left(  \phi_s(z)\right) dsdt +O(1/k^4).
		\end{align*} 
		
		From the equality
		\begin{equation}
			\sigma_{[e^{it\Lambda_0} S e^{-it\Lambda_0},S]}=-i\lbrace{\sigma_{e^{it\Lambda_0} S e^{-it\Lambda_0}}, \sigma_{S}\rbrace},
		\end{equation}
		where $\lbrace{\cdot,\cdot\rbrace}$ is the Poisson bracket defined through the canonical symplectic form of $T^*\S$, and the use of 
  Egorov's theorem:
		\begin{equation*}
			\sigma_{e^{it\Lambda_0} S e^{-it\Lambda_0}}=\phi^*_t(\sigma_{S}),
		\end{equation*}
we obtain that

		\begin{align*}
			\frakS_{\frac{1}{2}[F, S]^\ave}([z],k)=&\frac{-1}{8\pi^2 k^3}\int_0^{2\pi}t\int_0^{2\pi}\lbrace\phi^*_{t+s}(\sigma_{S}),\phi^*_s(\sigma_{S})\rbrace([z])ds\,dt +O(1/k^4)\\
			=&\frac{-1}{32\pi^2 k^3}\int_0^{2\pi}t\int_0^{2\pi}\lbrace\phi^*_{t+s}(q/\vert\xi\vert),\phi^*_s(q/\vert\xi\vert)\rbrace([z])ds\,dt +O(1/k^4),
		\end{align*}
		where we have used the equation:
		\begin{equation*}
			\phi^*_t\left(\lbrace f,g\rbrace\right)=\lbrace \phi^*_t(f),\phi^*_t (g)\rbrace.
		\end{equation*}
	\end{proof}

\end{corollary}


\section{Proofs of the main results}\label{seccuatro}
\newcommand{\Q}{\widetilde{Q}}
As claimed in \cite{Uribe} and further explained in Appendix \ref{apendiceA},
for each $\widetilde{Q}\in\frakR$ of order $d$ there exists a sequence of functions $q_j\in C^\infty(\calO)$, $j=0,1,\ldots$ such that,
as $k\to\infty$
\begin{equation}\label{eq:asymptSymbol}
	\frakS_{\widetilde{Q}}(\cdot, k)\sim \sum_{j=0}^\infty k^{d-j}\, q_j(\cdot).
\end{equation}
Moreover, $q_0$ is equal to the usual  principal symbol of $\widetilde{Q}$, restricted to $\calZ$  and then regarded as a function on $\calO$.
With this notation, one has:

\begin{theorem}\label{thm:GeneralBands} 
	Let $\widetilde{Q}\in \frakR$ be a zeroth-order self-adjoint operator, and let
	\begin{equation}\label{}
		\frakS_{\widetilde{Q}} \sim \sum_{j=0}^\infty q_j k^{-j}
	\end{equation}
	be the full expansion of its Berezin symbol.  Then, for any $f\in C^\infty(\bbR)$ there is
	an asymptotic expansion of the rescaled traces
	\begin{equation}\label{}
		\frac{1}{d_k}\tr\left(f(\widetilde{Q})|_{\calH_k}\right) \sim \sum_{j=0}^\infty \beta_j(f)k^{-j}
	\end{equation}
	where the $\beta_j$ are given for $j=0,1,2$ by:
	\begin{equation}\label{eq:betaZero}
		\beta_0(f) = \int_\calO f(q_0)\, d[w],
	\end{equation}
	\begin{equation}\label{eq:betaOne}
		\beta_1(f) = \int_\calO f'(q_0)\left( \frac 14 \Delta_{\calO}(q_0) + q_1 \right)\, \, d[w],
	\end{equation}
	and
	\begin{equation}\label{eq:betaTwo}
		\beta_2(f) = \int_\calO f''(q_0)\, \Gamma_2\, d[w] + \int_\calO f'(q_0)\, \Gamma_1\, d[w], 
	\end{equation}
	where the $\Gamma_i$ are given by (\ref{eq:Gamma1}) and (\ref{eq:Gamma2}) and $\Delta_{\calO}$ is the Laplacian of $\calO$ determined by the K\"ahler structure of $\calO$ which will be explained below in Appendix \ref{apendiceA}.
\end{theorem}

Our main results follow from this theorem and the results of Sections \ref{sec calc_Berezin} and \ref{seccionPromedio}.

\begin{remark}
    The above expression for $\beta_2$ is different from the one in
    \cite{Uribe}.  We have not been able to reconstruct a derivation of 
    the latter.  However, we have not been able to find a contradiction
    either.  (For example, both expressions are true in the case 
    $\Lambda_0\Q = -i(x_1\partial x_2 - x_2\partial x_1)$, which 
    can be computed explicitly).
    As will become apparent in the proof, there are many
    ways of writing $\beta_2(f)$ as an integral of an expression involving
    the $q_j$ and their derivatives.
\end{remark}

\subsection{The covariant symbol calculus of $\frakR$}
To prove Theorem \ref{thm:GeneralBands}, we will use the full symbol calculus of the Berezin symbol.  We begin by recalling the main result of \cite{Uribe} (see also Appendix A):
\begin{theorem}\label{thm:ExistenceSymbolCalc}
		
		There exists a sequence $D_\ell,\, \ell=0,1,\ldots$ of bilinear differential operators on functions on $\calO$ such that, 
		$\forall A,\,B\in\frakR$ of order  $d_A$ and $d_B$ respectively, 
		\begin{equation}\label{eq:product law}
			\frakS_{A\circ B} \sim k^{d_A+d_B}\sum_{j=0}^\infty k^{-j} \sum_{\ell+m+n=j} D_\ell(a_m, b_n).
		\end{equation}
		The $D_i$ are of order $i$ in each entry.
		$D_0(a,b) = ab$, and $D_1,\, D_2$ will be given below.
  \end{theorem}
	\begin{remark}
	The expression (\ref{eq:product law}) defines what is called  a star product on $C^\infty(\calO)[[\h]]$, see \cite{Uribe}
 \end{remark}

\medskip
To describe the operators $D_1,\, D_2$ we identify $\calO$ with a unit sphere, 
and introduce a complex stereographic coordinate $z$ on $\calO$.
For future reference we now list a few formulas for operators and 
other basic objects on $\calO$.  Letting $\nu(z) = 1+|z|^2$,
the Laplace-Beltrami operator on $\calO$ is
\begin{equation}\label{eq:LapO}
	\Delta_{\calO} =  -\nu^2 \frac{\partial^2\ }{\del z\del\zbar}, \quad z=x+iy,\quad 
	\del_z = \frac 12\left(\del_x - i\del_y\right),
\end{equation}
the Riemannian metric is $\frac{4}{\nu^{2}}(dx^2+dy^2)$, and the gradient of $f:\calO\to\bbR$ is 
\begin{equation}\label{eq:gradO}
	\nabla_{\calO} f = \frac{\nu^2}{4}\left( f_x\partial_x + f_y\partial_y \right) = 
	\frac{\nu^2}{2}\left(f_z\partial_{\zbar} + f_{\zbar}\partial_z\right).
\end{equation}
The expression
\begin{equation}\label{eq:normGradSq}
	\norm{f}^2 = \frac{\nu^2}{4}\left( f_x^2 + f_y^2\right) = \nu^2\,f_z f_{\zbar}
\end{equation}
will appear frequently in our computations.
The symplectic form on $\calO$ (arising from reduction of $T^*\S$) is
\begin{equation}\label{eq:symplForm}
	\omega = \frac{2i}{\nu^2}\, dz\wedge d\zbar =  \frac{4}{\nu^2}\, dx\wedge dy .
\end{equation}
It satisfies $\int_\calO \omega = 4\pi$, and since it is
rotationaly invariant, the normalized area form must be $d[w] = \frac{1}{4\pi}\omega$.

With respect to $\omega$, the Hamilton field of $f:\calO\to\bbR$ is 
\begin{equation}\label{eq:HamFieldO}
	\xi_f =\frac{\nu^2}{4} \left( -f_y\partial_x+ f_x\partial_y\right) , \qquad 	\omega(\cdot, \xi_f) = df(\cdot).
\end{equation}

\bigskip
Going back to the operators appearing in the star product above, we claim that, 
\begin{equation}\label{eq:Duno}
	D_1(f,g) = \frac{\nu(z)^2}{2}\,\frac{\partial f}{\partial z}\, \frac{\partial g}{\partial \zbar},
\end{equation}
and 
\begin{equation}\label{eq:Ddos}
	8	D_2(f,g) = \nu^4\, 	\frac{\partial^2 f}{\partial z^2}\, \frac{\partial^2 g}{\partial \zbar^2} +
	2\nu^3 \left(   \zbar \frac{\partial  f}{\partial z}\frac{\partial^2 g}{\partial \zbar^2} +  z\frac{\partial^2  f}{\partial z^2}\frac{\partial g}{\partial \zbar}	\right) + 
	4|z|^2\nu^2\,\frac{\partial f}{\partial z}\, \frac{\partial g}{\partial \zbar}. 
\end{equation}
In  Appendix \ref{apendiceA} we explain how these operators arise.  In particular, note that $D_1$ is a
(complex) vector field in each entry, which has the following intrinsic interpretation:
\begin{lemma}\label{lem:D_1}
	The operator $D_1$ is given by:
	\[
	D_1(f,g) = \frac{1}{2} \inner{\nabla_{\calO} f}{\nabla_{\calO} g} + \frac{1}{2i} \PB{f}{g}_{\calO}.
	\]
Equivalently,
	\begin{equation}\label{eq:D1isNabla}
		D_1(f, g)+ D_1(g, f) = \inner{\nabla_{\calO} f}{\nabla_{\calO} q} \quad\text{and}\quad 
		D_1(f, g)- D_1(g, f) = -i\PB{f}{q}_{\calO}.
	\end{equation}
 where $\PB{f}{q}_{\calO}$ is the Poisson bracket of $f$ and $g$ determined by the symplectic form $\omega$ on $\calO$

One also has:
	\[
	D_1(f,g) = \sqrt{-1}\times \left[\text{the }(1,0) \text{ component of } \xi_f\ \text{applied to }g\right].
	\]
%
\end{lemma}
Note that the second identity in (\ref{eq:D1isNabla}) says that the star product of our calculus
is in the direction of the Poisson bracket of $\calO$.

\subsection{The symbol of the exponential}

Let $\Q \in\frakR$ 
be self-adjoint and of order zero.  
The $\beta_i$ in Theorem \ref{thm:GeneralBands} are compactly-supported
distributions.  We will in fact compute their inverse Fourier transform
$\calF^{-1}(\beta_j)$, which is to say, we will compute the asymptotics 
\begin{equation}
    \frac{1}{2k+1}\tr\left[ e^{it\Q}|_{\calH_k}  \right] \sim 2\pi \sum_{j\geq 0}
\calF^{-1}(\beta_j)(t)\, k^{-j}
\end{equation}
as $k\to\infty$.  Here the exponential $e^{it\Q}$ is defined by the spectral theorem.
(This is related to (\ref{eq:calcSymb})).
It is known that, for each $t$, $e^{it\Q}$ is a zeroth order $\Psi$DO and it clearly commutes with $\Delta_{\S}$, and therefore it 
is in $\frakR$.  We let
\begin{equation}\label{eq:symbolExp}
	\frakS_{e^{it\Q}}(\cdot, k) \sim \sum_{j=0}^\infty a_j(t, \cdot)\, k^{-j},
\end{equation}
and will compute the first few $a_j$, in terms of the full Berezin symbol of
$\Q$, by analyzing the equation that
the exponential $e^{it\Q}$ satisfies.

\begin{lemma}
	The functons $a_j$ satisfy:  $a_0 = e^{itq_0}$ and
	\begin{equation}\label{eq:laEcuaExponencial}
		\forall j\geq 1\qquad	\dot{a_j} = iq_0 a_j + F_j, \qquad a_j(0) = 0
	\end{equation}
	where $F_j$ is the sum over non-negative indices
	\begin{equation}\label{eq:Fj}
		F_j = i \sum_{\substack{n+m+r=j\\ r <j}} D_{n}(q_{m}, a_{r}).
	\end{equation}
	(This holds for any star product.)
\end{lemma}

\begin{proof}   Letting $k^{-1} =\h$, we have (using star product notation)  
	\[
	-i \dot\frakS_{e^{it\Q}}(\cdot, k)  = \left(\sum_{m=0}^\infty q_m\,\h^m\right)\star 
	\left(\sum_{r=0}^\infty a_r\,\h^r\right) = 
	\sum_{m, r=0}^\infty (q_m\star a_r)	\h^{m+r} =
	\sum_{m, r, n=0}^\infty D_n (q_m,  a_r)	\h^{m+r+n} .
	\]
	It follows that the coefficient of $\h^j$ is
	\[
	\sum_{m+n+r=j} D_n (q_m,  a_r).
	\]
	There is exactly one term in this sum involving $a_j$, namely $q_0a_j$.  
	Peeling off this term from the sum leaves the desired expression for $F_j$.
	
\end{proof}

\begin{proposition}
	For each $j\geq 1$ the solution to (\ref{eq:laEcuaExponencial}) is of the form
	\begin{equation}\label{eq:ajExponencial}
		a_j = e^{itq_0}\Phi_j,
	\end{equation}
	where  $\Phi_j\in C^\infty(\bbR_t\times\calO )$ satisfies $\Phi_j|_{t=0} = 0$.
	Moreover, $\Phi_j$ is a polynomial of degree $2j$ in $t$ with coefficients functions on $\calO$.
\end{proposition}
\begin{proof}
	Substituting the ansatz (\ref{eq:ajExponencial}) into (\ref{eq:laEcuaExponencial}), we see that the latter is equivalent to
	\begin{equation}\label{eq:dotPhi}
		\dot\Phi_j = e^{-itq_0}F_j.
	\end{equation}
	We procceed by strong induction.  Assume that $a_r$ has the desired form for all $r<j$, and analyze the terms appearing in $F_j$,
	namely $D_{n}(q_{m}, a_r)$ with $n+m+r=j$.  The operator $D_{n}(q_{m}, \cdot)$ is a differential operator in the $\calO$ variables of degree
	$n$.  Since $a_r= e^{itq_0}\Phi_r$, the largest power of $t$ in $D_{n}(q_{m}, a_r)$ arises from terms where all derivatives fall
	on the factor $e^{itq_0}$, times the leading term in $\Phi_r$.  Since there are at most $n$ derivatives,
	\[
	D_{n}(q_{m}, a_r) = e^{itq_0} \calF
	\]
	where $\calF$ is a polynomial in $t$ of degree at most $n+2r$ with coefficients smooth functions on $\calO$.
	Now in the expression for $F_j$, $n+2r = j+r$ which is maximal if $r=j-1$.  Therefore,
	$\dot\Phi_j$ is a polynomial in $t$ of degree $2j-1$, and $\Phi_j$ itself has degree $2j$ in $t$.
\end{proof}

	\subsubsection{Computation of $a_1$}

Using that $D_1(q_0, \cdot)$ is a vector field,
		\[
		-iF_1 = D_1(q_0, e^{itq_0}) + q_1 e^{itq_0} = e^{itq_0}\left(itD_1(q_0, q_0) + q_1\right).  
		\]
		Therefore $\dot\Phi_1 = -t D_1(q_0, q_0) + iq_1$, and  $\Phi_1 = -\frac{t^2}{2}D_1(q_0,q_0)+itq_1$, so that
		\begin{equation}\label{eq:D_1Prelim}
			a_1 = e^{itq_0}\left(-\frac{t^2}{2}D_1(q_0,q_0)+itq_1\right).
		\end{equation}
		In view of Lemma \ref{lem:D_1}, we can conclude that
		\begin{equation}\label{eq:aUno}
a_1 = e^{itq_0}\left(-\frac{t^2}{4}\norm{\nabla q_0}^2+itq_1\right).
		\end{equation}

\subsubsection{Computation of  $a_2$}  From (\ref{eq:Fj}) we have that
		\begin{equation}\label{eq:F2Start}
			-iF_2 = q_1 a_1 + q_2 e^{itq_0}+  D_1(q_1, e^{itq_0}) + D_1(q_0, a_1)+ D_2(q_0, e^{itq_0}) .
		\end{equation}
		We now compute individual terms.  First,
		\begin{equation}\label{eq:FirstTermsF2}
			 q_1 a_1 + q_2 e^{itq_0}= e^{itq_0}\left( -q_1\frac{t^2}{4}\norm{\nabla q_0}^2+itq_1^2+q_2 \right).
		\end{equation}
		Next, since $D_1$ is a vector field in each entry
\begin{equation}\label{eq:aTermD1}
		D_1(q_1, e^{itq_0}) = ite^{itq_0} D_1(q_1, q_0).
\end{equation}
	The fourth term in (\ref{eq:F2Start}) is more complicated. Using (\ref{eq:D_1Prelim}) we have that
\begin{equation}\label{eq:MoreComplicated}
		D_1(q_0, a_1) = -\frac{t^2}{2} D_1(q_0, e^{itq_0}D_1(q_0, q_0)) + it D_1(q_0, e^{itq_0}q_1).
\end{equation}
Expanding the first term of this expression we get
\[
-\frac{t^2}{2} e^{itq_0}\left[D_1(q_0, D_1(q_0, q_0)) + it D_1(q_0, q_0)^2\right],
\]
while the second equals
\[
ite^{itq_0}\left[D_1(q_0, q_1)+ it q_1 D_1(q_0, q_0)\right].
\]
Going back to (\ref{eq:MoreComplicated}), combining and arranging terms by powers of $t$ we obtain
\begin{equation}\label{eq:CombineArrange}
	e^{-itq_0} D_1(q_0, a_1) = \frac{t^3}{2i} D_1(q_0, q_0)^2  -
	 t^2\left[\frac 12 D_1(q_0, D_1(q_0, q_0)) +q_1 D_1(q_0, q_0) \right] + it D_1(q_0, q_1) =
\end{equation}
\begin{equation}\label{}
	= \frac{t^3}{8i}\norm{\nabla q_0}^4 -t^2\left[\frac 14 D_1(q_0, \norm{\nabla q_0}^2) +\frac 12 q_1 \norm{\nabla q_0}^2 \right]
	+ it D_1(q_0, q_1),
\end{equation}
where we have used Lemma \ref{lem:D_1}.
Using (\ref{eq:FirstTermsF2}, \ref{eq:aTermD1}) and said lemma, we can
summarize the current state of the calculation as follows:
\begin{lemma}\label{lem:StateAffairs}
$-ie^{-itq_0}F_2$ is equal to the sum
\begin{equation}\label{eq:StateAffairs1}
\frac{t^3}{8i}\norm{\nabla q_0}^4
-\frac{t^2}{4}\left[D_1(q_0, \norm{\nabla q_0}^2) +3 q_1 \norm{\nabla q_0}^2 \right]
+ it \left[q_1^2 + \inner{\nabla q_0}{\nabla q_1}\right] + q_2 + e^{-itq_0}D_2(q_0, e^{itq_0}).	
\end{equation}
The term $e^{-itq_0}D_2(q_0, e^{itq_0})$ is a polynomial in $t$ of degree two.  Specifically,
\begin{equation}\label{eq:aD2term}
	8 e^{-itq_0}D_2(q_0, e^{itq_0}) = 
	-t^2 \nu^3 (q_{\zbar})^2\left( \nu q_{zz} + 2\zbar q_z	\right) + 8it D_2(q_0,q_0),
\end{equation}
where we have let $q_{\zbar} = \frac{\partial\ }{\partial \zbar} q_0$, etc.
\end{lemma}
\begin{proof}
	The only non-proved statement is (\ref{eq:aD2term}), which is a direct calculation starting with (\ref{eq:Ddos}).
\end{proof}

\medskip
To continue, we analyze the term $D_1(q_0, \norm{\nabla q_0}^2) $ in coordinates.  The starting point is
\begin{equation}\label{}
	\norm{\nabla q_0}^2 = 2D_1(q_0, q_0) = \nu^2 q_z q_{\zbar}.
\end{equation}
Then a short computation (using (\ref{eq:Duno})) shows that
\begin{equation}\label{}
D_1(q_0, \norm{\nabla q_0}^2) = -\frac 12 \norm{\nabla q_0}^2\Delta q_0 + \nu^3\, (q_z)^2
\left(\frac \nu 2 q_{\zbar \zbar} +z q_{\zbar }\right).
\end{equation}
The second term will combine with the first term on the right-hand side of (\ref{eq:aD2term}) to yield:
\begin{lemma}\label{lem:UpdateCalc}
$-ie^{-itq_0}F_2$ is equal to the sum
\begin{equation}\label{eq:StateAffairs}
	\frac{t^3}{8i}\norm{\nabla q_0}^4
	-\frac{t^2}{4}\left[\Upsilon -\frac 12 \norm{\nabla q_0}^2 \Delta q_0 +3 q_1 \norm{\nabla q_0}^2\right]
	+ it \left[q_1^2 + \inner{\nabla q_0}{\nabla q_1}+D_2(q_0, q_0)\right] + q_2 .	
\end{equation}
where
\begin{equation}\label{eq:Upsilon}
	\Upsilon := \frac{\nu^3}{2}\left[(q_{\zbar}^2)(\nu q_{zz}+ 2\zbar q_z) + \text{C.C.}\right]
\end{equation}
(Here C.C. stands for the complex conjugate of the expression preceeding it; note that $q_0$ and $\nu$ are real).
\end{lemma}
Next we interpret the expression $\Upsilon$ intrinsically:
				
\begin{lemma}
\[
\Upsilon = \nabla q_0(\norm{\nabla q_0}^2) +  \norm{\nabla q_0}^2\Delta q_0.
\]
\end{lemma}
\begin{proof}
The proof is a computation in coordinates.  Using the second identity in (\ref{eq:gradO})
\[
\nabla q_0(\norm{\nabla q_0}^2) = \frac{\nu^2}{2}\left(q_z\partial_{\zbar} + \text{C.C.}\right)(\nu^2  q_z\, q_{\zbar}) = 
 \frac{\nu^2}{2} q_z \left(2\nu z q_z q_{\zbar} + \nu^2 q_{z\zbar}q_{\zbar} + \nu^2 q_z q_{\zbar\,\zbar}\right) + \text{C.C.} = 
\]
\[
= \frac{\nu^3}{2} \left[  q_z^2(2zq_{\zbar}+\nu q_{\zbar\,\zbar}) + \text{C.C.}   \right] + \nu^4 q_{z\zbar} q_z q_{\zbar} 
= \Upsilon -\Delta(q_0)\,\norm{\nabla q_0}^2.
\]
\end{proof}

Combining the previous lemmas, and referring to (\ref{eq:dotPhi}), we obtain:
\begin{alignat}{2}\label{eq:PhiDot}
\dot\Phi_2	= e^{-itq_0}F_2 = & \frac{t^3}{8}\norm{\nabla q_0}^4
		-i\frac{t^2}{4}\left[\nabla q_0(\norm{\nabla q_0}^2) + \frac 12 \norm{\nabla q_0}^2\Delta q_0 +3 q_1 \norm{\nabla q_0}^2\right] \\ \notag
	&-t \left[q_1^2 + \inner{\nabla q_0}{\nabla q_1}+D_2(q_0, q_0)\right] + iq_2 .	
\end{alignat}

Finally, recall that the function $\Phi_2 = e^{-itq_0}a_2$ is the primitive of (\ref{eq:PhiDot}) with respect to $t$ that vanishes at $t=0$ (see (\ref{eq:ajExponencial})). 
We summarize:
\begin{proposition}
The coefficient $a_2 $ in the expansion of the covariant symbol of $e^{it\widetilde{Q}}$ satisfies
\begin{alignat}{2}\label{eq:aDos} 
e^{-itq_0}a_2 = & \frac{t^4}{32}\norm{\nabla q_0}^4 
-i\frac{t^3}{12}\left[\nabla q_0(\norm{\nabla q_0}^2) + \frac 12 \norm{\nabla q_0}^2\Delta q_0 +3 q_1 \norm{\nabla q_0}^2\right]\\ 
\notag &-\frac{t^2}{2} \left[q_1^2 + \inner{\nabla q_0}{\nabla q_1}+D_2(q_0, q_0)\right] + itq_2 .	
\end{alignat}
\end{proposition}

\subsection{Computation of $\beta_i,\, i=0, 1, 2$}	\label{seccionFinal}

In this section we finalize the computation of the first three invariants $\beta_j$.  
With the notation (\ref{eq:symbolExp}), the inverse Fourier transform of $\beta_j$ for all $j$ (considered now as a distribution)
is
\begin{equation}\label{}
	\calF^{-1} (\beta_j )(t) = \frac{1}{2\pi}\int_\calO a_j(t,[w])\, d[w].
\end{equation}
This means that for any test function $f$, if $\calF(f) (s) = \int_\bbR e^{-ist}\, f(t)\, dt$ denotes
its  Fourier transform,
\begin{equation}\label{eq:inverseFT}
	(\beta_j,  f)  = 	(\calF^{-1}(\beta_j), \calF(f) ) =   \frac{1}{2\pi}\iint_{\bbR\times\calO}a_j(t,[w])\, \calF(f)(t)\, d[w]\, dt.
\end{equation}
In particular, changing the order of integration gives
\begin{equation}\label{}
(\beta_0,  f)  =  \frac{1}{2\pi} \iint_{\bbR\times\calO} e^{itq_0}\, \calF( f)(t)\, d[w]\, dt = \int_\calO f(q_0)\, d[w].
\end{equation}
(Taking into account that $q_0 = \frac 12 \hat{q}$, we get the first item in Theorem \ref{thm:Main}.)

\medskip  
The identity (\ref{eq:inverseFT}) will be used to compute $\beta_i$ , $i=1,2$.  From the formulas for $a_1,\, a_2$
of the previous section it would appear that $\beta_i$ is a distribution of order $2i$.  However, we will see that
the order of $\beta_i$ can be reduced to $i$ (for $i=1,2$) by integration by parts, by means of the following lemma:

\begin{lemma}\label{lem:byParts}
Let $F\in C^\infty(\bbR)$, and $u, v \in C^\infty(\calO)$.  Then
\[
\int_\calO v F''(u) \norm{\nabla u}^2\, d[w] = \int_\calO F'(u)\, v\Delta(u)\, d[w] - \int_\calO F(u)\Delta(v)\, d[w].
\]
\end{lemma}
\begin{proof}
		For any $F\in C^{\infty}(\bbR)$ and $u\in C^\infty(\calO)$, one has that
	\begin{equation}\label{eq:IdentidadLap}
		\Delta (F(u)) = F'(u)\Delta(u) - F''(u) \norm{\nabla u}^2.
	\end{equation}
	To obtain the desired result, multiply by $v$, integrate, and use the symmetry of $\Delta$.
	
\end{proof}

\subsubsection{Computation of $\beta_1$}

Substituting in (\ref{eq:inverseFT}) the expression for $a_1$ that we found in (\ref{eq:aUno}) yields
\[
(\beta_1, f) =  \frac{1}{2\pi}\iint_{\bbR\times\calO}e^{itq_0}\left(-\frac{t^2}{4}\norm{\nabla q_0}^2 + itq_1\right)\, \calF(f)(t)\, d[w]\, dt =
\]
\begin{equation}\label{}
	= \frac 14 \int_\calO f''(q_0)\, \norm{\nabla q_0}^2\, d[w] + \int_\calO f'(q_0)\, q_1\, d[w].
\end{equation}
Using Lemma \ref{lem:byParts} with $F' =f$ and $v\equiv 1$ we obtain (\ref{eq:betaOne}).


\subsubsection{Computation of $\beta_2$}

Let us write
\begin{equation}\label{eq:a2EnPotencias}
	a_2 = e^{itq_0}\sum_{j=1}^4 t^j \Psi_j,
\end{equation}
where the $\Psi_j \in C^\infty(\calO)$ are given in (\ref{eq:aDos}).  Then, by (\ref{eq:inverseFT}), for every $f\in C^\infty(\bbR)$,
\begin{equation}\label{eq:beta2EnPotencias}
	(\beta_2, f) = \sum_{j=1}^4 (-i)^j\, \int_\calO f^{(j)}(q_0)\, \Psi_j\, d[w].
\end{equation}
The apparent order of $\beta_2$ (the maximum number of derivatives of $f$ that are needed to evaluate $(\beta_2, f)$)
can be lowered integrating by parts certain terms, as follows.

\begin{lemma}
\begin{equation}\label{eq:ByParts3b}
	\int_\calO \norm{\nabla q_0}^2 f^{(3)}(q_0) \Delta q_0  \, d[w] = \int_\calO \left[f''(q_0)(\Delta q_0)^2 - f'(q_0)\Delta^2 q_0\right].
\end{equation}
\end{lemma}
\begin{proof}
	Apply Lemma \ref{lem:byParts} with $F=f'$ and $v = \Delta q_0$.
%
\end{proof}

\begin{lemma}
\begin{equation}\label{eq:ByParts4}
	\int_\calO f^{(4)}(q_0)\, \norm{\nabla q_0}^4\, d[w] = 
	\int_\calO f''(q_0)\left[ (\Delta q_0)^2- \Delta(\norm{\nabla q_0}^2)\right]\, d[w] - 
	\int_\calO f'(q_0)\Delta^2( q_0)\, d[w].
\end{equation}
\end{lemma}
\begin{proof}
Using (\ref{eq:IdentidadLap})  with $F = f''$, one can derive that
\[
 f^{(4)}(q_0)\, \norm{\nabla q_0}^4 = \norm{\nabla q_0}^2\left[ f^{(3)}(q_0) \Delta q_0 - \Delta(f''(q_0))\right].
\]
We can now quote (\ref{eq:ByParts3b})  and the symmetry of $\Delta$ to conclude. 
\end{proof}

\begin{lemma}
	\begin{equation}\label{eq:ByParts3a}
		\int_\calO f^{(3)}\, {\nabla q_0}(\norm{\nabla q_0}^2)\, d[w] =  \int_\calO f''(q_0)\, \Delta(\norm{q_0}^2)\, d[w].
	\end{equation}
\end{lemma}
\begin{proof}
	For any $F\in C^{\infty}(\bbR)$ and $u\in C^\infty(\calO)$, the function
	\[
	\Delta (F(u)\norm{\nabla u}^2) = \norm{\nabla u}^2 \Delta (F(u)) + F(u) \Delta(\norm{\nabla u}^2) - 2F'(u)\inner{\nabla u}{\nabla\norm{\nabla u}^2}
	\]
	integrates to zero.  Using that $\Delta$ is symmetric, we obtain
	\[
	\int_\calO F(u) \Delta(\norm{\nabla u}^2))\, d[w] = \int_\calO F'(u) \nabla u(\norm{\nabla u}^2)\, d[w].
	\]
	Apply this with $F=f''$ and $u=q_0$.
	\end{proof}
	
	The final term  involving $f^{(3)}$ is dealt with similarly, using Lemma \ref{lem:byParts}:
\begin{lemma}
\begin{equation}\label{eq:ByParts3c}
	\int_\calO f^{(3)}(q_0) q_1 \norm{\nabla q_0}^2\, d[w] = \int_\calO \left[f''(q_0) q_1\Delta(q_0) - f'(q_0)\Delta q_1\right]\, d[w].
\end{equation}
\end{lemma}

We have now reduced the order of $\beta_2$ to two.
We now complete the calculation.  Referring to  (\ref{eq:aDos}), let us compute the summands in 
(\ref{eq:beta2EnPotencias}) individually, integrating by parts according to the previous lemmas.

The $j=4$ term is
\begin{equation}\label{eq:betaDosCuartaDerivada}
	\int_\calO f^{(4)} (q_0)\Psi_4 d[w] = \frac{1}{32}\int f^{(4)}(q_0) \norm{\nabla q_0}^4 d[w] = 
\end{equation}
\[
=\frac{1}{32}\int_\calO f''(q_0)\left[ (\Delta q_0)^2- \Delta(\norm{\nabla q_0}^2)\right]\, d[w] - \frac{1}{32}\int_\calO f'(q_0)\Delta^2(q_0)\, d[w].
\]
Next, the $j=3$ contribution:
\begin{equation}\label{eq:betaDosTerceraDerivada}
	i	\int_\calO f^{(3)} (q_0)\Psi_3 d[w] = \calA + \calB + \calC,\qquad\text{where}
\end{equation}
\begin{equation}\label{eq:betaDosTerceraA}
	\calA = \frac{1}{12}\int_\calO f^{(3)}(q_0) \nabla q_0 (\norm{\nabla q_0}^2) d[w] =  
	\frac{1}{12} \int_\calO f''(q_0) \Delta (\norm{\nabla q_0}^2) d[w],
\end{equation}
\begin{equation}\label{eq:betaDosTerceraB}
	\calB = \frac{1}{24} \int_\calO f^{(3)}(q_0) \Delta q_0 \norm{\nabla q_0}^2 d[w] =  \frac{1}{24} \int_\calO f''(q_0) (\Delta  q_0)^2 d[w] -
	 \frac{1}{24} \int_\calO f'(q_0) \Delta^2 q_0 \, d[w],
\end{equation}
and
\begin{equation}\label{eq:betaDosTerceraC}
	\calC = \frac 14 \int_\calO f^{(3)}(q_0) q_1 \norm{\nabla q_0}^2 d[w] = 
	\frac 14\int_\calO f''(q_0)q_1 \Delta(q_0) d[w] - 
	\frac 14\int_\calO f'(q_0) \Delta q_1\, d[w].
\end{equation}
The $j=2$ term is (no integration by parts)
\begin{equation}\label{eq:betaDosSegunda}
	-\int_\calO f''(q_0) \Psi_2 d[w] = \frac 12 \int_\calO f''(q_0)\left(q_1^2 + \inner{\nabla q_0}{\nabla q_1} + D_2(q_0, q_0)\right)\, d[w],
\end{equation}
and the $j=1$ term is simply
\begin{equation}\label{eq:betaDosPrimera}
	-i\int_\calO f'(q_0) \Psi_1 d[w] = \int_\calO f'(q_0) q_2\, d[w].
\end{equation}

To obtain $(\beta_2, f)$ we simply add (\ref{eq:betaDosCuartaDerivada}, \ref{eq:betaDosTerceraA}, \ref{eq:betaDosTerceraB}, \ref{eq:betaDosTerceraC},
\ref{eq:betaDosSegunda}) and (\ref{eq:betaDosPrimera}).  The result is of the form
\begin{equation}\label{eq:CombinarBetaDos}
	(\beta_2, f) = \int_\calO f''(q_0)\, \Gamma_2\, d[w] + \int_\calO f'(q_0)\, \Gamma_1\, d[w] 
\end{equation}
where: 
\begin{equation}\label{eq:Gamma1}
	\Gamma_1 = q_2 -\frac 14 \Delta q_1-\frac{7}{96} \Delta^2 q_0,
\end{equation}
and

\begin{equation}\label{eq:Gamma2}
	\Gamma_2 = \frac{7}{96}(\Delta q_0)^2 + \frac{5}{96} \Delta(\norm{\nabla q_0}^2) + \frac{1}{4} q_1\Delta q_0
+ \frac 12\left(q_1^2 + \inner{\nabla q_0}{\nabla q_1} + D_2(q_0, q_0)\right).
\end{equation}





\bigskip

\subsection{The end of the proof}
Theorem \ref{thm:Main} follows directly from \ref{thm:GeneralBands} and the results from Sections \ref{sec calc_Berezin} and \ref{seccionPromedio}.
More specifically, we take $\Q = \Lambda_0 Q$ and use the
expressions for $q_j, \,j=0,1,2$ found in those sections.

Next, assume $q|_{\S}$ is an odd function.  Then the operator $Q$ of Section \ref{seccionPromedio} is of order (-2), and we consider
\[
\Q = \Lambda^2 Q.
\]
By the results of Section \ref{seccionPromedio}, the covariant symbol of $\Q$ satisfies
\[
\frakS_{\widetilde{Q}} \sim \sum_{j=0}^\infty \tilde q_j k^{-j}
\]
where
\begin{equation}\label{eq:qzeroOdd}
	\tilde q_0 = -\frac 14 \calI(\partial_r q)
\end{equation}
and
\begin{equation}\label{eq:qoneOdd}
	\tilde q_1 =
 \frac 18 \calI \left(2q^2+5\partial_rq+\partial^2_rq-\frac{1}{2}\partial_r\Delta_{\S}q\right)
		+ W,
	\end{equation}
	where $W:\calO\to\bbC$ is the function given  by
	\begin{equation}\label{}
		W([z])=\frac{-1}{32\pi^2}\int_0^{2\pi}t\int_0^{2\pi}\lbrace\phi^*_{t+s}(q/\vert\xi\vert),\phi^*_s(q/\vert\xi\vert)\rbrace(z)ds\,dt.
	\end{equation}

With this at hand, Theorem \ref{thm:casoImpar} follows from Theorem \ref{thm:GeneralBands}.

\newpage

\appendix

\section{The ring $\frakR$,  Berezin-Toeplitz operators, and the Berezin calculus}	\label{apendiceA}

In this Appendix we show that one can 
identify the ring $\frakR$ with the ring of Berezin-Toeplitz operators over
the space $\calO$, building on results of Guillemin, \cite{Gu1984}.  Under this identification, 
which is new, the symbol calculus of \cite{Uribe} 
is the same as the covariant symbol calculus of
Berezin-Toeplitz operators as developed by L. Charles in \cite{LCharles}.


\subsection{The Hardy space of  $\calZ$ and Toeplitz operators}

We begin by summarizing some results from \cite{Gu1984}, Sections 5 and 6.
Recall that we are identifying 
\[
\calZ \cong \left\{z\in\bbC^3\:|\: z\cdot z=0, \ \norm{z}^2=2\right\}.
\]
Therefore
\begin{equation}\label{}
	\calZ = \partial\calW, \quad \calW = \left\{z\in\bbC^3\:|\: z\cdot z=0, \ \norm{z}^2< 2\right\}.
\end{equation}
The space $\calW$ is a strictly pseudoconvex domain of the quadric
\[
\calQ = \{z\in\bbC^3\:|\: z\cdot z=0\},
\]
with defining function
\[
\rho(z) = \frac 12\norm{z}^2-1.
\]
One can check that, with the identification above, $\vartheta=\Im\delbar \rho$ is identified with the canonical
one-form on $T^*S^2$ pulled-back to the unit (co)tangent bundle $\calZ$.

The action of SO$(3)$ extends complex-linearly to $\bbC^3$, and it preserves $\calW$ and $\calZ$.
The action on $\calZ$ is the standard action on the unit tangent bundle of $\S$.  We endow
$\calZ$ with the SO$(3)$ normalized 
invariant measure (denoted $dz$).
We will denote by $\calH(\calZ)$ the $L^2$
Hardy space of $\calZ$, that is, the $L^2$ closure of boundary values of holomorphic functions
on $\calW$.    
Therefore, SO$(3)$ is represented unitarily in $\calH(\calZ)$.  The decomposition 
of the Hardy space of $\calZ$ into isotypical subspaces is
\begin{equation}\label{eq:deco}
	\calH(\calZ) = \bigoplus_{k=0}^\infty \calH(\calZ)_k,
\end{equation}
where
$\calH(\calZ)_k$ consists of the restrictions to $\calZ$ of polynomials $\psi$ in $z$
homogeneous of degree $k$ and satisfying $\sum_{j=1}^3 \frac{\partial^2\psi}{\partial z_j^2}=0$.
Clearly then
\begin{equation}\label{}
	\forall k\quad \calH_k \cong \calH(\calZ)_k,
\end{equation}
as both spaces are isomorphic to the space of harmonic complex homogeneous polynomials of degree k
in three variables.  

\medskip
More formally, for each $k$ one can define a linear isomorphism 
\begin{equation}\label{}
	P_k: \calH_k\to\calH(\calZ)_k
\end{equation} 
which is simply analytic continuation from the variables $(x_1,x_2,x_3)\in\bbR^3$ to $(z_1, z_2,z_3)\in\bbC^3$.
$P_k$ and its adjoint $P_k^*$ are equivariant, so by Schur's lemma $P_k^*P_k = a_k I$ where $a_k>0$ is a positive
constant.  It follows that $\frac{1}{\sqrt{a_k}}P_k: \calH_k\to\calH(\calZ)_k$ is an equivariant unitary map
(a surjective isometry).
To obtain a map in the opposite direction, let $p: \calZ\to \S$ be the (cotangent) projection, $p(z) = \Re z$.  The fibers of 
$p$ are unit circles (with respect to the Euclidean structure of the (co)tangent spaces of $\S$).
Let
\begin{equation}\label{eq:pushForwardOp}
	p_*: C^\infty(\calZ)\cap\calH(\calZ)\to C^\infty(\S)
\end{equation}
be the operator of integration over the fibers of $p$ with respect to the induced measure.
Note that $p_*$ is also equivariant with respect to the action of the rotation group.  Therefore,
for each $k$, $p_*$ maps $\calH(\calZ)_k$ into $\calH_k$, and the compositions
$p_*\circ P_k$, $P_k\circ p_*$ must be multiples of the identity (by Schur's lemma again):
\begin{equation}\label{eq:pstarCircPk}
	\forall k\in\bbN\ \exists \tau_k\not= 0\qquad p_*\circ P_k = \tau_k I_{\calH_k}.
\end{equation}
(This is equation (6.14) in \cite{Gu1984}.)
Composing on the right by $P_k^*$, we obtain
\begin{equation}\label{eq:pstarAsAdjoint}
	p_*|_{\calH(\calZ)_k} = \frac{\tau_k}{a_k} P_k^*.
\end{equation}

\begin{theorem} \cite[Theorem 6.2]{Gu1984}  The operator 
	$p_*$ extends to a continuous isomorphism $p_*: \calH(\calZ)\to H^2_{1/4}(\S)$ where 
	$H^2_{1/4}(\S)$ is the Sobolev space consisting of functions $f\in L^2(\S)$ such that
	\[
	\norm{f}^2_{1/4}:= \sum_{k=0}^\infty (k+1)^{1/2} \norm{f_k}^2 < \infty,
	\]
	where $f = \sum_k f_k$ is the decomposition of $f$ into spherical harmonics.
\end{theorem}
\begin{corollary}
	The operator $p_*$ in (\ref{eq:pushForwardOp}) is a bijection.
	
\end{corollary}

For future reference, we introduce the functions in $\calH(\calZ)_k$ that
correspond to the coherent states $\alpha_z^k$.    For each $k\in\bbN$ and $z\in\calZ$,
let
\begin{equation}\label{eq:defVartheta}
	\varpi_z: \calZ\to\bbC, \qquad \varpi_z (w) := w\cdot z.
\end{equation}

\begin{proposition}\label{prop:CstatesCorrespond}
	For each $k\in\bbN$ and $z\in\calZ$, $\varpi_z^k \in \calH(\calZ)_k$.  
	In fact $p_*\varpi_w^k = \tau_k \alpha_z^k$.
\end{proposition}
\begin{proof}
	It is clear that $\varpi^k_z$ is the analytic continuation of $\alpha_z^k$, i.e. 
	$\varpi^k_z = P_k(\alpha_z^k)$.  Now apply $p_*$ to both sides and use
	(\ref{eq:pstarCircPk}).
\end{proof}

\medskip
Next, let
\begin{equation}\label{}
	\Pi: L^2(\calZ)\to\calH(\calZ)	
\end{equation}
be the orthogonal projector (the Szeg\H{o} projector).  
We recall that a Toeplitz operator on $\calH(\calZ) $ is an operator of the 
form
\[
T: \calH(\calZ)\to \calH(\calZ), \qquad T = \Pi \widetilde{Q}|_{\calH(\calZ)}
\]
where $\widetilde{Q}$ is a (classical) $\Psi$DO on $\calZ$.
By definition, 
the symbol of $T$ is  the function $\sigma_T: \calZ\to\bbC$ obtained by evaluating
the symbol of $\widetilde{Q}$ on the contact form $\eta\in\Omega^1(\calZ)$,
\begin{equation}\label{}
	\eta = \Im \delbar \rho.
\end{equation}

For our purposes,
the main results of \cite{Gu1984} can be summarized as follows:

\begin{theorem}\label{thm:Correspondence}\cite[Theorems 5.2 and 6.4]{Gu1984}
	For every pseudodifferential operator $Q$ on $\S$ there exists a unique
	Toeplitz operator $T$ on $\calZ$ such that
	\begin{equation}\label{}
		Q\circ p_* = p_*\circ T,
	\end{equation}
	and conversely.
	Moreover, the symbols of $T$ and of $Q$ agree on $\calZ$, 
	and $Q\in\calR$ iff $T$ commutes with the action of $S^1$ on $\calZ$ 
	given by complex multiplication.
\end{theorem}

Thinking of $\calZ$ as a subset of $\bbC^3$ (see (\ref{eq:IdentifyUnitTgt})), let us define (using the cross product)
\begin{equation}\label{eq:frakD}
	\frakD := \left\{(z,w)\in \calZ\times\calZ\:|\: \ z\times\zbar = -w\times\wbar \right\}.
\end{equation}
It follows from Lemma \ref{lem:GeoFlowCpx} that the set  $\frakD$ defined in (\ref{eq:frakD})
is ${\mathbb S}^1$ invariant separately in each 
variable, and,  under the identificaton of $\calO$ with a two-sphere,  it projects onto the subset
\[
\widetilde{\frakD} \subset\calO\times\calO,\qquad \widetilde{\frakD} = \left\{ \text{pairs of antipodal points}\right\}. 
\]

The following relates the covariant amplitudes of $Q$ and of $T$.
\begin{corollary}\label{cor:BerSymbsSame}
	Let $Q$ and $T$ be as in the previous theorem.  Then, for every $z,w\in\calZ$ such that $(z,w)\not\in\frakD$,
	\begin{equation}\label{eq:relCovSymbols}
		\frac{\inner{Q \alpha_z^k}{\alpha_w^k}}{\inner{\alpha_z^k}{\alpha_w^k}} =
		\frac{\inner{T\varpi_z^k}{\varpi_w^k}}{\inner{\varpi_z^k}{\varpi_w^k}}.
	\end{equation}
\end{corollary}
\begin{remark}
	We will see below that the denominators above do not vanish iff  $(z,w)\not\in \frakD$.
\end{remark}
\begin{proof}  By Proposition \ref{prop:CstatesCorrespond}
	\begin{equation}\label{eq:itsProof}
		\inner{Q\alpha_z^k}{\alpha_w^k}  = \tau_k^{-2}\inner{Qp_*\varpi_z^k}{p_*\varpi_w^k} = 
		a_k^{-2}\inner{P_k T\varpi_z^k}{P_k\varpi_w^k} = a_k^{-1}\inner{T\varpi_z^k}{\varpi_w^k}.
	\end{equation}
	
	Similarly, taking $Q$ and $T$ equal to the identity, we see that $\inner{\alpha_z^k}{\alpha_w^k}  =  a_k^{-1}\inner{\varpi_z^k}{\varpi_w^k}$,
	and (\ref{eq:relCovSymbols}) follows.
\end{proof}

\subsection{$\frakR$ and Berezin-Toeplitz operators on $\calO$}

Next, we recognize $\calZ$ as the unit circle bundle of a Hermitian complex
line bundle over $\calO$.

\subsubsection{The K\"ahler structure of $\calO$}\label{seccionKahlerO}

First we discuss a natural K\"ahler structure on $\calO$.
Consider $\bbC^3$ with its canonical K\"ahler form,
\[
\omega_0 = \frac i2 \sum_{j=1}^3 dz_j\wedge d\zbar_j, \qquad z = \langle z_1, z_2, z_3 \rangle.
\] 
Then the time $t$ map of the Hamilton flow of the function $\tilde\Phi(z) = \frac{1}{2}|z|^2$ is: $z\mapsto e^{it}z$.
Now the quadric 
\[
\calQ = \{z\in\bbC^3\setminus\{0\}\:|\: z\cdot z=0\}
\]
is a complex submanifold of $\bbC^3\setminus\{0\}$, and the pull-back $\omega_\calQ:= \iota^*\omega_0$, where
$\iota: \calQ\hookrightarrow \bbC^3$ is the inclusion, is a K\"ahler form on $\calQ$.  Let  
$\Phi: \calQ\to \bbR$ the composition $\tilde\Phi\circ\iota$.  Since $\calQ$ is invariant under 
multiplication by complex numbers,  $\Phi$ is still the Hamiltonian of the action of $S^1\subset\bbC$ on $(\calQ, \omega_\calQ)$.

One can check that $1\in\bbR$ is a regular value of $\Phi$, and the $S^1$ action is free on 
\begin{equation}\label{}
	\Phi^{-1}(1) \cong \calZ.
\end{equation}
(We recall that the above isomorphism is $\calZ\ni (\xi, \eta)\mapsto \xi+i\eta\in \Phi^{-1}(1)$.)
Since, by definition, $\calO \cong \calZ/S^1$,
by one of the results of \cite{GS}, $\calO$ inherits a K\"ahler structure, that we will describe more concretely next.

For every $z\in\calQ$, note that 
\[
T_z\calQ = \{\zeta\in\bbC^3\:|\: \zeta\cdot z = 0\},
\]
and, if $z\in\calZ$
\begin{equation}\label{}
	T_z\calZ = \{\zeta\in\bbC^3\:|\: \zeta\cdot z = 0\text{ and } \Re(z\cdot\zetabar)=0\}.
\end{equation}
Denote by $G_z := iz\in T_z\calQ$ the generator of the circle action on $\calQ$.  Let
\begin{equation}\label{}
	W_z := \left(\bbC G_z\right)^\bot\cap T_z\calQ,   \quad\text{the Hermitian orthogonal of } G_z\ \text{in}\ T_z\calQ.
\end{equation}
More explicitly,
\begin{equation}\label{eq:WzExplicitly}
	W_z = \left\{ \zeta \in\bbC^3\:|\: \zeta\cdot z = 0 = \zeta\cdot\zbar\right\}.
\end{equation}
Note that $W_z$ is a complex-linear subspace of $\bbC^3$.  It is the maximal complex subspace of $T_z\calZ$.

For each $z\in\calQ$,  the differential of the projection $\pi: \calZ\to\calO$ restricts to an isomorphism
\begin{equation}\label{eq:horizIsomorphism}
	d\pi_z:  W_z\cong T_{[z]}\calO.
\end{equation}
Moreover, since the $S^1$ action is by unitary maps, its differential maps $W_z$  to $W_{e^it z}$ for each $t$,
preserving the complex and Hermitian structures.    

Although we will not use this here, we note that the spaces $W_z, z\in\calZ$ are the horizontal subspaces of a connection on the principal circle bundle $\calZ\to \calO$.

The following is now immediate from these considerations:
\begin{lemma}
	For each $[z]\in \calO$, there exists a unique K\"ahler structure on $T_{[z]}\calO$, i.e. a pair $(\omega_{[z]}, J_z)$ of a
	symplectic form and a compatible linear complex structure on $T_{[z]}\calO$, such that the maps (\ref{eq:horizIsomorphism}) are
	isomorphisms of K\"ahler vector spaces.
\end{lemma}
Since $\calO$ has (real) dimension 2, the resulting two-form $\omega$ is automatically closed, and the complex structure on $\calO$ is 
integrable (there is no need to appeal to the general theory of \cite{GS}).  We have thus obtained a K\"ahler 
structure on $\calO$, which is invariant under the action of $SO(3)$.

\subsubsection{Quantizing $\calO$}
\begin{proposition}
	Let $\calL^*\to \calO$ be the complex  Hermitian line bundle associated to the 
	circle bundle $\calZ\to\calO$ and the identity character $S^1\to S^1$.   Let
	$\calD\subset\calL^*$ be the unit disk bundle.  Then $\calD$ is complex-analytically 
	isomorphic to the-blow up of $\calW$ at $0\in\calW$.
	Moreover, the Hardy space of $\calZ$ as the boundary of $\calD$ is $\calH(\calZ)$
	(the Hardy space of $\calZ$ as the boundary of $\calW$).
\end{proposition}
\begin{proof}
	By definition,
	\[
	\calL^* = \calZ\times\bbC/\sim\ ,\quad\text{where}\quad (e^{i\theta}z, \lambda) \sim (z, e^{i\theta}\lambda),
	\]
	and $\calD = \calZ\times D/\sim$, where $D\subset\bbC^1$ is the unit disk.  
	Then the map $\calD\to\calW$ given by
	\[
	\calD\ni [(z,\lambda)]\mapsto \lambda z\in\calW
	\]
	is the desired blow-up map of $0\in\calW$.   Note in particular that the fiber of this map over $0\in\calW$ is $\calZ/S^1 = \calO$.
	The statement about the Hardy spaces follows from the fact that any function
	analytic on $\calD\setminus \calO$ extends to $\calD$.
\end{proof}

The notation implies that we are interested in the dual bundle $\calL\to \calO$.  
The base $\calO$ inherits an SO$(3)$-invariant  K\"ahler structure, and 
$\calL\to\calQ$ is a holomorphic line bundle.  By a general
tautology in the theory of line bundles, there is a natural unitary isomorphism
\begin{equation}\label{eq:tautologyBdles}
	\forall k\in\bbN\quad \calH(\calZ)_k \cong H^0(\calO, \calL^k)
\end{equation}
between $ \calH(\calZ)_k$ and the space of holomorphic sections of the $k$-th tensor
power of $\calL$.

\medskip
To summarize the results of this section:
\begin{corollary}
	The correspondence of Theorem \ref{thm:Correspondence} establishes an isomorphism 
	between the ring $\calR$ and the ring of Berezin-Toeplitz operators on $\calL\to\calO$.
	
	Moreover, the full Berezin symbol of $Q\in\calR$ is equal to the covariant symbol
	of the corresponding Berezin-Toeplitz operator $T_Q$.
\end{corollary}
The last statement is simply Corollary \ref{cor:BerSymbsSame}.

\medskip

\subsection{The operators $D_j$}
Having identified the Berezin symbol of operators in $\calR$
with the covariant symbol of corresponding Berezin-Toeplitz
operators, the existence of a symbol calculus to all orders
for the Berezin symbol follows from the calculus of 
covariant symbols.
Indeed it is known (see \cite{LCharles}, Section 4) that the covariant symbols of B-T operators have an associated
	star product, which gives the asymptotic expansion of the symbol of the composition.  For our purposes we need an explicit description of the
	first three bi-differential operators $D_j$
 in the covariant star product, which we will now compute (though the first two are universally known, see Proposition
	4 in \cite{LCharles}). 

\subsubsection{Kernels}

We begin by recalling basic facts on covariant symbols of operators in the sense of Berezin, \cite{B},
adapted to the current setting.  

Fix a positive integer $k$.   Then, by the irreducibility of the representation of $SO(3)$ in $\calH_k$, one has
that (c.f. Lemma 2.2 in \cite{Uribe})
\begin{equation}\label{}
	\forall \psi\in \calH_k  \qquad \psi = (2k+1)\int_\calZ \frac{\inner{\psi}{\alpha_z^k}}{\inner{\alpha_z^k}{\alpha_z^k}}\alpha_z^k\, dz,
\end{equation}
where $dz$ is the invariant measure on $\calZ$ of total mass equal to one.
To compare with the notation in \cite{B}, the family of vectors $\{e_z\}$ given by
\begin{equation}\label{}
	e_z := \frac{\sqrt{2k+1}}{\norm{\alpha_z^k}}\alpha_{\zbar}^k, \quad z\in\calZ	
\end{equation}
satisfies 	
\begin{equation}\label{eq:overcomplete}
	\forall \psi\in \calH_k \ \psi = \int_\calZ \inner{\psi}{e_z}\, e_z\, dz,
\end{equation}
i.e. it 
	is an ``overcomplete" family.  
	
\begin{lemma}\label{lem:crucialRelKernels}  Let $Q \in \calR$ and $T:\calH(\calZ)\to\calH(\calZ)$ be the corresponding Toeplitz operator (see Theorem \ref{thm:Correspondence}). 
For $k\in\bbN$, let $\Pi_k:L^2(\calZ)\to \calH(\calZ)_k$ be the orthogonal projection.
Then  there exists constants $c_k$ such that the Schwartz kernel of $\Pi_k T\Pi_k$ satisfies
\begin{equation}\label{}
\calK_{\Pi_kT\Pi_k}(z,w)=c_k\inner{Qe_w}{e_z}. 
\end{equation}
\end{lemma} 
\begin{proof} 
	In view of (\ref{eq:itsProof}), it suffices to show that for some constant $c_k$, 
	\[
	\calK_{\Pi_kT\Pi_k}(z,w)=c_k\inner{T\varpi^k_{\wbar}}{\varpi^k_{\zbar}}.
	\]
	We begin by showing that there exists a constant $c_k$ such that
	\begin{equation}\label{eq:reproKernel}
\forall f\in L^2(\calZ),\ z\in\calZ\quad		\Pi_k (f)(z) = c_k\inner{f}{\varpi^k_{\zbar}}.
	\end{equation}
	To this end, define an operator $\widetilde{\Pi}_k$ by:
	\begin{equation}\label{}
		\forall z\in\calZ\qquad \widetilde{\Pi}_k(f)(z) := \inner{f}{\varpi^k_{\zbar}} = \int_\calZ f(w) \, (\wbar\cdot z)^k\, dw.
	\end{equation}
	It is clear that $\widetilde{\Pi}_k(f)\in\calH(\calZ)_k$, and $\forall g\in \text{SO}(3)$ 
	\[
	\widetilde{\Pi}_k(f)(g^{-1} z) = \int_\calZ f(w) \, (g\wbar\cdot z)^k\, dw = \int_\calZ f(g^{-1}w) \, (\wbar\cdot z)^k\, dw 
	= \int_\calZ( g\cdot f)(w) \, (g\wbar\cdot z)^k\, dw,
	\]
	where we have used that $g$ is real.  That is, $\widetilde{\Pi}_k$ is equivariant.  It is clear that $\widetilde{\Pi}_k$ is zero on
	$\calH(\calZ)_k^\bot$ and is non-zero, so by Schur's lemma we can conclude (\ref{eq:reproKernel}) for some non-zero constant $c_k$.
	
	Let $f\in \calH(\calZ)_k$.  Then $\forall z\in\calZ$
	\[
f(z) = \Pi_k(f)(z) =   c_k\int_\calZ f(w) \, (\wbar\cdot z)^k\, dw = c_k\int_\calZ f(w)\, \varpi_{\wbar}^k(z)\, dw,
	\]
	or $f = c_k\int_\calZ f(w)\varpi_{\wbar}^k\, dw$.  Applying $T$ on both sides ($T$ preserves $\calH(\calZ)_k$ since it corresponds
	to a $\Psi$DO on $\S$ that commutes with the Laplacian) we obtain
	\[
	T(f)(z) = c_k\int_\calZ f(w) \, T(\varpi_{\wbar}^k)(z)\, dw =  c_k^2\int_\calZ f(w) \inner{T\varpi_{\wbar}^k}{\varpi_{\zbar^k}}\, dw.
	\]
	This shows that the Schwartz kernel of $T$ restricted to $\calH(\calZ))k$ is $c_k^2 \inner{T\varpi_{\wbar}^k}{\varpi_{\zbar^k}}$.
	
\end{proof}

 For any $k$ and any
	linear map $A:\calH_k\to\calH_k$, let us define the function
	\begin{equation}\label{eq:kerneles}
		{\bf A}: \calZ\times\calZ\setminus \frakD \to\bbC, \qquad
		{\bf A}(z,w):= \frac{\inner{A\alpha_z^k}{\alpha_w^k}}{\inner{\alpha_z^k}{\alpha_w^k}}.
	\end{equation}
	Note that ${\bf A}(z,w)$ is separately $S^1$ invariant in each variable.  Therefore, it 
	descends to a function 
	\begin{equation}\label{eq:reducedKernel}
		{\bf A}([z],[w]): \calO\times\calO\setminus\widetilde\frakD\to\bbC
	\end{equation}
	whose restriction to the diagonal  is the covariant symbol of $A$:
	\begin{equation}\label{}
		\frakS_A:\calO\to\bbC,\quad 	\frakS_A([z]) = {\bf A}([z],[z]).
	\end{equation}



 

\medskip
For operators in $\calR$, the kernels ${\bf A}$ defined in  (\ref{eq:kerneles}) depend on $k$ and have  the following asymptotic behavior:
	\begin{theorem}\label{thm:CharlesKernels}(\cite{LCharles})
		Let $A\in\calR$ be of order zero.  Then the kernel function (\ref{eq:kerneles}) associated with ${\bf A}$
  is a symbol in $(z,w)$: there exists an asymptotic expansion as $k\to\infty$ in the $C^\infty$ topology 
		\begin{equation}\label{eq:TAmplitudeExpansion}
			{\bf A}(z,w;k)\sim \sum_{j=0}^\infty k^{-j} {\bf A}_j(z,w).
		\end{equation}
		Moreover, for all $j$
		\begin{equation}\label{eq:almostHolomorphic}
			\delbar_{[z]} { {\bf A}_j}(z,w)\ \text{and}  \ \partial_{[w]}{ {\bf A}_j}(z,w) \quad\text{vanish to infinite order on the diagonal }\
    \{z=w\}.
		\end{equation}

	\end{theorem}
	\begin{proof} By Lemma \ref{lem:crucialRelKernels}, the
		 function ${\bf A}$ is the Schwartz kernel of the B-T operator with multiplier $A$
		divided by the Schwartz kernel of the projection. Theorem 2 in \cite{LCharles}, 
  describes the Schwartz kernels of Berezin-Toeplitz operators, inlcuding the projection $\Pi$ itself. 
Our function ${\bf A}$ is the ratio of two functions $a$ 
  appearing in equation (2) of Charles' paper.  Therefore the theorem just cited implies the desired properties for ${\bf A}$.
	\end{proof}
	
	\begin{remark}
		In particular we can restrict (\ref{eq:TAmplitudeExpansion}) to obtain that the covariant symbol $\frakS_A$ has an asymptotic expansion
		\[
		\frakS_A(z,k)\sim \sum_{j=0}^\infty k^{-j}a_j(z)
		\]
		in the $C^\infty$ topology.
	\end{remark}

 \subsubsection{Composition}
We now turn to the symbol of the composition. 

	\begin{proposition}\label{prop:ExactCompoCovariant}(\cite[(1.11)]{B}) For each $k$ and any given 
 linear maps $A, B:\calH_k\to\calH_k$, the covariant symbol of their
		composition is
		\begin{equation}\label{eq:covariantCompo}
			\frakS_{A\circ B}([z]) = (2k+1) \int_\calO {\bf B}([z],[w])\, {\bf A}([w],[z])\, \frac{|\inner{\alpha_z^k}{\alpha_w^k}|^2}{\norm{\alpha_z^k}^4}\, d[w]
		\end{equation}
		where the measure on $\calO$ has been normalized.
	\end{proposition}
	
	\begin{remark}
		The integrand is not singular at $[w]=-[z]$, because the singularities in ${\bf B}(z,w)\, {\bf A}(w,z)$
		exactly cancel with $|\inner{\alpha_z^k}{\alpha_w^k}|^2$.  Explicitly, (\ref{eq:covariantCompo}) is equivalent to
		\begin{equation}\label{}
			\frakS_{A\circ B}([z]) = \frac{2k+1}{\norm{\alpha_z^k}^4}\,  \int_\calO  \inner{B(\alpha_{[z]}^k)}{\alpha_{[w]}^k}\, 
			\inner{A(\alpha_{[w]}^k)}{\alpha_{[z]}^k}\,d[w].
		\end{equation}
	\end{remark}
	
 The previous proposition leads us to introduce:
	\begin{definition}
		The Berezin kernel is the sequence of functions $\frakB_k:\calO\times\calO\to\R$ given by
		\begin{equation}\label{eq:BerKer}
			\frakB_k (p,q) :=(2k+1) \frac{|\inner{\alpha_z^k}{\alpha_w^k}|^2}{\norm{\alpha_z^k}^4}, \quad\text{where } p=[z],\ q=[w].
		\end{equation}
	\end{definition}
	In the model $\calO\cong\calS^2$, it is known (Lemma 6.3 in \cite{Uribe}) that
	\begin{equation}\label{eq:BerKerExplicit}
		\forall k\in\bbN\qquad	\frakB_k (p,q) = (2k+1)\left(\frac{1+\cos\theta(p,q)}{2}\right)^{2k},
	\end{equation}
	where $\theta(p,q)$ is the angle between the position vectors of $p, q\in\calO$.
	With this notation, (\ref{eq:covariantCompo}) can be expressed as:
	\begin{equation}\label{eq:covariantCompoDos}
		\forall p\in\calO\qquad \frakS_{A\circ B}(p) = \int_\calO \frakB_k(p, q) {\bf B}_k (p,q)\, {\bf A}_k (q,p)\, dq.
	\end{equation}
	
	We recall that, for each $k$, the operator 
	\begin{equation}\label{eq:BerTrans}
		\frakB_k:  C^\infty(\calO)\to C^\infty(\calO), \qquad \frakB_k(f)(p) = 	\int_\calO \frakB_k(p, q) f(q)\, dq
	\end{equation}
	is called the {\em Berezin transform}.  In addition to appearing in the composition formula
	(\ref{eq:covariantCompoDos}), 
	$\frakB_k(f)$ is the covariant symbol of the Berezin-Toeplitz operator with multiplier $f$
	(\cite[equation (1.12)]{B}).  (We refer to \cite{Pol} for another interesting interpretation of the
	Berezin transform, as generator of a Markov process.)

	\begin{proposition}\label{prop:AsymptBKernel}(\cite[Section 6]{Uribe})
		There exists a sequence of linear differential operators on $\calO$, ${E_j}$, such that for all 
		$f\in C^\infty(\calO)$
		\begin{equation}\label{eq:expansionBerTrans}
			\int_\calO \frakB_k(p, q) f(q)\, d[q] \sim \sum_{j=0}^\infty k^{-j} E_j(f)(p)
		\end{equation}
		as $k\to\infty$.  Moreover, $E_0=I$,
		\begin{equation}\label{eq:primerosEs}
			E_1 =  -\frac 12 \Delta_\calO \quad\text{and}\quad E_2	=\frac 18 \Delta_\calO^2 + \frac 14 \Delta_\calO,
		\end{equation}
		where $\Delta_\calO$ denotes the Laplace-Beltrami operator acting on functions on $\calO$.
		The integral in (\ref{eq:expansionBerTrans}) is with respect to the normalized invariant measure introduced above.
	\end{proposition}
	\begin{remark}
		All the operators $E_j$ are functions of $\Delta_\calO$, as they must, by equivariance with respect to 
		the $SO(3)$ action.  
	\end{remark}
	\begin{remark}
		The formula for $E_0$ and $E_1$ hold in the general context 
		of Berezin-Toeplitz quantization  
		\cite[equation (1.2)]{Karabegov} 
	\end{remark}

	
	\begin{corollary}
		Let ${\bf A}([z],[w]),\ {\bf B}([z],[w])$ be two $k$-independent functions satisfying the property (\ref{eq:almostHolomorphic}), and
  let $D_j$ be the bi-differential operators such that
\begin{equation}
    E_j({\bf A}(z,w)\, {\bf B}(w,z))|_{w=z} = D_j(A, B)
\end{equation}
  where $A([z]) = A([z],[z])$, and similarly for $B$.
  Then
		\begin{equation}\label{eq:starProdOne}
			\int_\calO \frakB_k([z], [w]) {\bf A}([z],[w])\, {\bf B}([w],[z])\, d[w] \sim \sum_{j=0}^\infty k^{-j} D_j(A, B)([z]).
		\end{equation}

	\end{corollary}
 \begin{proof}
     Apply the previous proposition to the function 
     $f([w]) = {\bf A}([z],[w])\, {\bf B}([w],[z])$. 
\end{proof}
	

\begin{remark}
    This result extends to $k$-dependent kernels
${\bf A}(k,[z],[w]),\ {\bf B}(k,[z],[w])$ with the properties stated in Theorem \ref{thm:CharlesKernels}:
One has
\begin{equation}
\int_\calO \frakB_k([z], [w]) {\bf A}(k, [z],[w])\, {\bf B}(k,[w],[z])\, d[w] \sim \sum_{j,\ell,m =0}^\infty k^{-j-\ell-m} 
D_j(A_\ell, B_m)([z])
\end{equation}
because the expansions (\ref{eq:TAmplitudeExpansion}) are in the $C^\infty$ topology.
By Proposition \ref{prop:ExactCompoCovariant},
this result precisely says that the operators $D_j$ are the ones giving the star product of the covariant
symbol calculus.
\end{remark}

    Finally, we observe that in a complex stereographic coordinate $w$ on $\calO$, 
	\begin{equation}\label{}
		\Delta_\calO = - (1+|w|^2)^2\, \frac{\partial^2\ }{\partial w \partial\wbar}
	\end{equation}
	Let ${\bf A}(z,w),\ {\bf B}(z,w)$ be two $k$-independent functions satisfying the property (\ref{eq:almostHolomorphic}).  Then $\forall z$
	\begin{equation}\label{}
		\Delta_\calO ({\bf A}(z,w)\, {\bf B}(w,z))|_{w=z} = - (1+|w|^2)^2\, \left(\frac{\partial {\bf A}(z,w)}{\partial \wbar}|_{w=z}\right)
		\, \left(\frac{\partial {\bf B}(w,z)}{\partial w}|_{w=z}\right)
	\end{equation}
 	and similarly for higher powers of $\Delta$.  Moreover: 
 \begin{lemma}
      \begin{equation}
    \frac{\partial {\bf A}(z,w)}{\partial \wbar}|_{w=z} = \frac{\partial A}{\partial \zbar}(z),\qquad \text{where}\ A(z) = {\bf A}(z,z).
 \end{equation}
 and similarly for $\frac{\partial {\bf B}(w,z)}{\partial w}|_{w=z}$.
 \end{lemma}
 \begin{proof}
     This is a simple argument using Taylor series:  We can write
     \[
{\bf A}(z,w) = \sum_{p,q}c_{pq} z^p\wbar^q + R
     \]
     where $R$ is a function that vanishes to very high order on the diagonal.  Then
     \[
  \frac{\partial\ }{\partial \wbar}  {\bf A}(z,w)|_{w=z}   = \sum_{p,q}q c_{p,q} z^p\zbar^{q-1} = \frac{\partial\ }{\partial \zbar} \sum_{p,q}c_{pq} z^p\zbar^q.
     \]
 \end{proof}

 From this it follows that $D_j$ differentiates the first entry in the $(0,1)$ direction and the second entry in the $(1,0)$ direction.


\section{Calculations for Section \ref{sec calc_Berezin} } \label{apendiceB}

We divide this appendix in three parts

\subsection{$L^2$ norms of the coherent states:}
\begin{itemize}
 
 \item[i)]

$$\Vert \alpha_z^k\Vert_{L^2(\S)}^2=2\pi B(k+1,1/2)\sim k^{-1/2}$$
\item[ii)]  
\begin{equation} \label{normas de coherentes} \Vert \alpha_z^k\Vert_{L^2(\B)}^2=\frac{\pi}{k+1} B(k+2,1/2))\sim k^{-3/2}, 
\end{equation}
\end{itemize}where $B(x,y)$ the Beta function.
To see \textit{(i)},   assume that $\xi=e_1$ and $\eta=e_2$. Then $\vert \alpha_z^k(y)\vert^2=(y_1^2+y_2^2)^k$.
We use the formula for integration on the sphere in dimension $3$ of functions  constant in parallels
\begin{equation*}
\int_\S f(y_3)d\sigma (y)=2\pi\int_{-1}^1 f(s) ds.
\end{equation*}
Then
\begin{align*}
\Vert \alpha_z^k\Vert_{L^2(\S)}^2=&2\pi\int_{-1}^1 (1-y_3^{2})^kdy_3=2\pi B(k+1,1/2),
\end{align*}
and \textit{(ii)} follows from \textit{(i)}.
\subsection{
Proof of Lemma \ref{Radon de pot de Delta}} To prove the lemma, first notice that for each $m\geq 1$ we can write in spherical coordinates
\begin{equation}\label{patron}
\Delta_{\S}^mq=\partial_\varphi^{2m}q+\sum_{i=0}^{m-1} P_i(\varphi)\partial_\varphi^{2i}q+\sum_{i=0}^{m-1} N_i(\varphi)\partial_\varphi^{2i+1}q+ \partial_\theta Mq,
\end{equation}
where the derivatives of odd order  of each $P_i(\pi/2)$ vanish; the derivatives of order even  of each $N_j(\pi/2)$ are zero and where $ M$  is a differential operator.
In fact,  since $ \Delta_{\S}q=\partial_\varphi^2+\cot(\varphi)\partial_\varphi+\frac{1}{\sin(\varphi)}\partial_\theta^2 q$

\begin{equation}
\widehat{\Delta_{\S}q}(z)= \widetilde{\Delta_{\S}q}(1,\pi/2)=\widetilde{\partial_\varphi^2 q}(1,\varphi)
\end{equation}

 Proceeding by induction in $m$, assume \eqref{patron} valid for  $m$ and  write

\begin{equation*}
\Delta_{\S}^{m+1}q=\left( \partial_\varphi^2+\cot(\varphi)\partial_\varphi+\frac{1}{\sin(\varphi)^2}\partial_\theta^2\right)\Delta_{\S}^m q.
\end{equation*}
with $\Delta_{\S}^m q$ as in \eqref{patron}.
  Then a long and easy calculation using that any derivative of odd order of $cot(\varphi)$ at $\varphi=\pi/2$ is zero and Leibnitz rule shows directly that \eqref{patron} holds for $m+1$.

Then evaluating \eqref{patron} at $\varphi=\pi/2$, $r=1$, and integrating with respect to $\theta$ in the interval $[0,2\pi]$ (noticing that $Mq$ is periodic in
$\theta$), we obtain 
\begin{equation*}
\widehat{\Delta_{\S}^m q}(z)=\widetilde{\partial^{2m}_\varphi q}(1,\pi/2)+\sum_{i=0}^{m-1}a_i\widetilde{\partial^{2i}_\varphi q}(1,\pi/2).
\end{equation*}
Finally, we can solve this lower triangular linear system for $\widetilde{\partial^{2i}_\varphi q}(1,\pi/2)$ and the proof of the lemma is complete.
We easily see in particular that 
\begin{equation*}
\widetilde{\partial_{\varphi}^4q}(1,\pi/2)=\widehat{\Delta_{\S }^2q}(z)+2\widehat{\Delta_{\S }q}(z).
\end{equation*}
 
 \subsection{Asymptotics related to  the Beta function} Using that (see for example \cite[Th. 4.3]{Ele})
 \begin{equation*}
 \frac{\Gamma(k+1/2)}{\Gamma(k)}=\sqrt{k}\left(1-\frac{1}{4}\frac{1}{2k}+\frac{191}{64} \frac{1}{(2k)^3}+O(k^4)\right),
 \end{equation*}
 so that
 \begin{equation*}\label{cocientes betas1}
\frac{1}{2\pi B(k+1,1/2)} =\frac{(k+1/2)\Gamma(k+1/2)}{2\pi^{3/2}k\Gamma(k)}
\end{equation*}
\begin{equation*}\label{cocientes betas2}
=\sqrt{\frac{k}{\pi}}\frac{1}{2\pi}\left(1+\frac{3}{4(2\pi)}-\frac{1}{4(2k)^2}+\frac{191}{64(2k)^3}+O(1/k^4)\right).
\end{equation*}

\section{Proof of Theorem \ref{teoremaclusters}}
For completeness, we give a proof of Theorem \ref{teoremaclusters}.  
 We first establish the following 
\begin{lemma}\label{lemaclusters} Let $B=\Lambda_0S + S\Lambda_0 +S^2$. The spectrum of the operator $\Lambda_q^2$ is contained in the union of intervals
$$\bigcup_{k=0}^{\infty} \; \left[k^2-\|B\|,k^2+\|B\|\right].$$  Moreover, for $k$ suffiiciently large, the interval  $\left[k^2-\|B\|,k^2+\|B\|\right]$ contains  $d_k$ eigenvalues of $\Lambda_q^2$,  
counted with multiplicities.
\end{lemma}

\begin{proof}
Let us write $\Lambda_q^2= A +B$,  where  $A=\Lambda_0^2$ and $B$ is a $\Psi$DO of order zero and then bounded. Consider  $z$  an element of the resolvent set $\rho(A)$ of the operator $A$. We write
\begin{equation}
\Lambda_q^2 -z = (A-z)(I + (A-z)^{-1}B)
\end{equation}

Then if the distance $d(z,\sigma(A))$ between $z$ and the spectrum $\sigma(A)$ of  $A$ satisfies $d(z,\sigma(A))\geq \|B\|$ then $\| (A-z)^{-1}B\|<1$. Thus we have that  $z$ must be in  the resolvent set of $\Lambda_q^2$ and  then  
$\sigma(\Lambda_q^2)\subset \cup_{k=0}^{\infty} \; \left[k^2-\|B\|,k^2+\|B\|\right]$.

For $k$ sufficiently large, let $P_k$ be the projector of the operator  $\Lambda_q^2$ associated to the interval $ \left[k^2-\|B\|,k^2+\|B\|\right]$.   Let ${\mathcal C}_k$ be a circle with radius $r_k=k/2$ and center $k^2$.  Then 
\begin{eqnarray}
\|  P_k  - \Pi_k\| &=& \left\| \frac{1}{2\pi\imath}\int_{z\in{\mathcal C}_k}  \left[ \left(\Lambda_q^2-z\right)^{-1} -   \left(A-z\right)^{-1}\right] dz  \right\|  \nonumber \\
&\leq& \|  \left(\Lambda_q^2-z\right)^{-1} \| \left(A-z\right)^{-1}\|\ \| B \| r_k = O(1/k)
\end{eqnarray}
Thus $\|  P_k  - \Pi_k\| <1$  for $k$ sufficiently large,  which implies that the dimension of the range of  $\Pi_k$ and $P_k$  must be the same (see \cite{Kato}). 
\end{proof}

\begin{proof}[ Proof of Theorem \ref{teoremaclusters}] From Lemma \ref{lemaclusters} we have that  there exist $k_0>0$ such that outside a fixed compact interval around the origin, all the eigenvalues of $\Lambda_q^2$ can be written as $k^2+\lambda_{k,j}$ with $k\geq{k_0}$  and $j=1,\ldots, d_k$ and $|\lambda_{k,j}|\leq\|B\|$.   Therefore, all the eigenvalues of $\Lambda_q$ outside a suitable compact interval around the origin can be written as $\sqrt{k^2+\lambda_{k,j}}=k + O(1/k)$. 
\end{proof}


\end{document}